\documentclass[12pt,letterpaper]{article}
\IfFileExists{lmodern.sty}{%
  \usepackage[T1]{fontenc}
  \usepackage{lmodern}
}{}
\IfFileExists{geometry.sty}{%
  \usepackage[margin=1in]{geometry}
}{%
  \setlength{\textwidth}{6.5in}
  \setlength{\textheight}{9in}
  \setlength{\oddsidemargin}{0pt}
  \setlength{\evensidemargin}{0pt}
  \setlength{\topmargin}{0pt}
  \addtolength{\topmargin}{-\headheight}
  \addtolength{\topmargin}{-\headsep}
}
\usepackage{amsmath,amssymb,amsthm}
\usepackage{xcolor}
\usepackage{float,needspace}
\definecolor{ReferenceBlue}{rgb}{0,0,1}
\IfFileExists{enumitem.sty}{%
  \usepackage{enumitem}
  \setlist[enumerate,1]{label=(\roman*),leftmargin=*}
}{%
  
}
\IfFileExists{mathrsfs.sty}{\usepackage{mathrsfs}}{%
  \providecommand{\mathscr}[1]{\mathcal{#1}}
}
\IfFileExists{titlesec.sty}{\usepackage{titlesec}}{}
\IfFileExists{microtype.sty}{%
  \IfFileExists{lmodern.sty}{\usepackage{microtype}}{%
    \usepackage[expansion=false]{microtype}
  }
}{}
\IfFileExists{hyperref.sty}{\usepackage[colorlinks=true,linkcolor=ReferenceBlue,citecolor=ReferenceBlue,urlcolor=ReferenceBlue,hyperfootnotes=false]{hyperref}}{%
  \providecommand{\href}[2]{#2}
  \providecommand{\texorpdfstring}[2]{#1}
  \providecommand{\hypersetup}[1]{}
}
\numberwithin{equation}{section}
\newtheorem{theorem}{Theorem}[section]
\newtheorem{lemma}[theorem]{Lemma}
\newtheorem{proposition}[theorem]{Proposition}
\newtheorem{corollary}[theorem]{Corollary}
\theoremstyle{definition}

\theoremstyle{remark}
\newtheorem{remark}[theorem]{Remark}
\numberwithin{figure}{section}
\numberwithin{table}{section}
\newcommand{\R}{\mathbb{R}}
\newcommand{\N}{\mathbb{N}}
\newcommand{\D}{D^{N,q}(\R^N)}
\newcommand{\norm}[1]{\lVert#1\rVert}
\newcommand{\Ph}{\Phi_{N,q,\beta}}
\newcommand{\A}{\mathcal{A}}
\newcommand{\J}{\mathcal{J}}
\newcommand{\Sa}{\mathcal{S}}
\newcommand{\dd}{\,\mathrm{d}}
\newcommand{\e}{\mathrm{e}}
\DeclareMathOperator{\diver}{div}
\DeclareMathOperator{\Hess}{Hess}
\allowdisplaybreaks[2]
\hypersetup{pdftitle={Extremals and Thresholds for Critical Singular Anisotropic Moser--Trudinger Inequalities},pdfauthor={Kaiwen Guo, Yanjun Liu}}
\title{Extremals and Thresholds for Critical Singular
Anisotropic Moser--Trudinger Inequalities}
\author{Kaiwen Guo\textsuperscript{a}, Yanjun Liu\textsuperscript{b,*}\\[1ex]
\small \textsuperscript{a}Mathematical Research Center, Shandong University,\\
\small Jinan, 250100, P.~R.~China\\
\small \textsuperscript{b}School of Mathematical Sciences, Chongqing Normal University,\\
\small Chongqing, 401331, P.~R.~China}
\date{}
\begin{document}
\maketitle
\begingroup
\renewcommand{\thefootnote}{\fnsymbol{footnote}}
\footnotetext[1]{\raggedright Email address: \texttt{kwguo@mail.nankai.edu.cn} (Kaiwen Guo).\\
Corresponding author: \texttt{liuyj@mail.nankai.edu.cn} (Yanjun Liu).}
\endgroup
\begin{abstract}
For $N\ge2$, $0<\beta<N$ and $q>1$, we study maximizers of the critical
singular anisotropic Moser--Trudinger integral
\[
 \int_{\R^N}\frac{\Phi_{N,q,\beta}
 (\lambda_N(1-\beta/N)|u|^{N/(N-1)})}{F^o(x)^\beta}\dd x,
 \qquad \norm{F(\nabla u)}_N^a+\norm{u}_q^b\le1,
\]
where $a>0$, $0<b\le N$ and $\Phi_{N,q,\beta}$ is the integrable Taylor
remainder. The supremum is attained for every $a>0$ when $b<N$.
For $b=N$ and $1<q<q_+$, the first retained Taylor term separates two regimes, with
\[
 q_-:=\frac{N^2(N-2)}{(N-1)(N-\beta)},\qquad
 q_+:=\frac{N^2}{N-\beta}.
\]
Attainment holds for every $a>0$ if $1<q<q_-$, while for
$q>1$, $q_-\le q<q_+$, there is a finite threshold $a_c>N$ and
attainment holds if and only if $0<a\le a_c$.
An exact Euclidean reduction shows that the threshold is independent of
the anisotropy. The proof combines critical--subcritical scaling with a
nonlinear Green-function concentration bound. Radial flux and Pohozaev
identities determine the first correction to the subcritical supremum,
and refined Green tests give the strict comparison needed for attainment
at the finite threshold, including in dimension two.
\end{abstract}
\begingroup\raggedright
\par
\endgroup


\section{Introduction and main results}\label{sec:1}
The Moser--Trudinger inequality is the limiting Sobolev embedding.
For a bounded smooth domain $\Omega\subset\R^N$, $N\ge2$, Trudinger
\cite{refTrudinger1967} established exponential integrability and Moser
\cite{ref1} identified the sharp constant:
\[
 \sup_{\substack{u\in W_0^{1,N}(\Omega)\\\norm{\nabla u}_N\le1}}
 \int_\Omega \e^{\alpha_N|u|^{N/(N-1)}}\dd x\le C_N|\Omega|,
 \qquad \alpha_N=N^{N/(N-1)}\omega_N^{1/(N-1)},
\]
where $\omega_N$ is the volume of the Euclidean unit ball.
Carleson and Chang \cite{refCarlesonChang1986} proved attainment on the
ball; Flucher \cite{refFlucher1992} treated general planar domains by
Green-function analysis. Their strict comparison with the concentration
level is also central here. Refined radial energy estimates, as in
Mancini and Martinazzi \cite{refManciniMartinazzi2017}, show how the
correction beyond the leading bubble can decide attainment.

On $\R^N$, spatial dilation introduces an additional loss of compactness.
Adachi and Tanaka \cite{refAdachiTanaka2000} proved sharp subcritical
inequalities normalized by the $L^N$ norm, while Ruf \cite{refRuf2005}
and Li and Ruf \cite{refLiRuf2008} established the critical inequalities
under the full Sobolev constraint. Ishiwata \cite{refIshiwata2011}
compared maximizing levels with vanishing sequences. Lam, Lu and Zhang
\cite{refLamLuZhang2019} used the critical--subcritical relation to study
attainment and symmetry; Ikoma, Ishiwata and Wadade
\cite{refIkomaIshiwataWadade2019} obtained parameter thresholds for
constraints with independent norm powers.

The singular weight $|x|^{-\beta}$ lowers the sharp coefficient to
$(1-\beta/N)\alpha_N$ \cite{refAdimurthiSandeep2007}.
Whole-space inequalities and extremals were obtained in
\cite{ref2,ref3,refWangChen2021}; Nguyen \cite{ref11} studied thresholds
under mixed-norm constraints. Exact-growth estimates balance the
critical exponential with a polynomial denominator: see Masmoudi and
Sani \cite{refMasmoudiSani2015} in the Euclidean case and Liu
\cite{refLiu2022} for the anisotropic counterpart and its maximizers.

For anisotropic energies, convex symmetrization
\cite{ref4} replaces Euclidean balls by Wulff balls. Sharp inequalities,
blow-up analysis and extremals were developed in
\cite{refWangXia2012,ref5,ref6}; lower-order improvements appear in
\cite{refXie2023,refLiu2024Potential}. In the singular setting, Liu
\cite{ref7} proved concentration--compactness, Lu, Shen, Xue and Zhu
\cite{ref10} obtained weighted isoperimetric inequalities and extremals,
and Li \cite{refLi2024AnisotropicSingular} treated an improved
whole-space inequality. Guo and Liu \cite{refGuoLiu2026Bounded} proved
bounded-domain attainment with an $L^p$ correction.

Here the lower-order exponent $q$ and the two constraint powers vary
independently. We determine the attainment ranges below and prove that
the threshold in the gradient exponent is independent of the anisotropy.
Its finiteness follows from the first correction to the subcritical
supremum; attainment at the threshold requires a strict comparison
with its limiting concentration branch.

To formulate the problem, let $F:\R^N\to[0,\infty)$ be even, convex and positively homogeneous of
degree one. We assume that $F\in C^2(\R^N\setminus\{0\})$, $F(\xi)>0$ for
$\xi\ne0$, and that $\Hess(F^2)$ is positive definite away from the origin.
Its polar norm is
\[
 F^{o}(x):=\sup_{\xi\ne0}\frac{x\cdot\xi}{F(\xi)}.
\]
For $r>0$, write $W_r:=\{x:F^{o}(x)<r\}$ and $\kappa_N:=|W_1|$.
The sharp and singular coefficients are
\[
 \lambda_N=N^{\frac{N}{N-1}}\kappa_N^{\frac1{N-1}},\qquad
 \delta:=1-\frac{\beta}{N},\qquad \lambda_\beta:=\delta\lambda_N.
\]
We write $H_{N-1}:=\sum_{j=1}^{N-1}j^{-1}$ and use $A\simeq B$
for two-sided bounds with positive constants independent of the varying
parameter. Following \cite{ref8}, $\D$ is the completion of
$C_c^\infty(\R^N)$ for the norm $\norm{\nabla u}_N+\norm{u}_q$,
equivalently $\norm{F(\nabla u)}_N+\norm{u}_q$.
For $0<\beta<N$ and $q>1$, define
\begin{equation}\label{eq:2}
 m=m(N,q,\beta):=\min\left\{j\in\N:j>\frac{q(N-1)}{N}\delta\right\}
\end{equation}
and $\Ph(s):=\sum_{j=m}^\infty s^j/j!$ for $s\ge0$.
The strict inequality in \eqref{eq:2} gives
\begin{equation}\label{eq:4}
 \frac{mN}{N-1}>q\delta.
\end{equation}
For $a,b>0$, define the critical supremum
\begin{equation}\label{eq:5}
 \A_{a,b}:=
 \sup_{\norm{F(\nabla u)}_N^a+\norm{u}_q^b\le1}
 \int_{\R^N}\frac{\Ph(\lambda_\beta|u|^{\frac{N}{N-1}})}{F^{o}(x)^\beta}\dd x.
\end{equation}
Write $\J(u)$ for the integral in \eqref{eq:5}.
Guo and Liu \cite{ref8} proved finiteness precisely for $b\le N$ and
an exact critical--subcritical relation. The existence and Wulff symmetry
of subcritical maximizers \cite{ref9} reduce critical attainment to the
endpoint behavior of a scalar variational function.

At $b=N$, the first nonzero Taylor term of $\Ph$ determines the two ranges
in our attainment result. Set
\[
q_-:=\frac{N^2(N-2)}{(N-1)(N-\beta)},\qquad q_+:=\frac{N^2}{N-\beta}.
\]
The three ranges are characterized by
\begin{equation}\label{eq:7}
\begin{aligned}
 q<q_-&\iff m\le N-2,\\
 q_-\le q<q_+&\iff m=N-1,\\
 q\ge q_+&\iff m\ge N.
\end{aligned}
\end{equation}
The strict inequality in \eqref{eq:2} places $q=q_-$ in the range $m=N-1$.

Liang and Xiong \cite[Theorem~1.2]{ref12} established a threshold theory
at $b=N$ for $q\ge q_+$, under a quantitative condition relating the
Wulff-ball volume to a Gagliardo--Nirenberg constant. At $b=N$, we study
the complementary range $1<q<q_+$. The case $m=N-1$ requires a first-order comparison:
the leading Taylor contribution and the normalization error have the same
logarithmic order, so the coefficient depending on $a$ must be retained.

\Needspace{7\baselineskip}
We first record attainment for $b<N$, which follows directly from the variational reduction.
The assumptions on $F$ stated above remain in force.
\begin{theorem}\label{thm:1.1}
Let $N\ge2$, $0<\beta<N$, $q>1$, and $0<b<N$.
For every $a>0$, $\A_{a,b}$ is attained by a nonnegative
Wulff-symmetric function.
\end{theorem}

\Needspace{7\baselineskip}
At $b=N$, concentration may persist. The first Taylor index separates
attainment for all $a>0$ from a threshold in $a$.
\begin{theorem}\label{thm:1.2}
Let $N\ge2$, $0<\beta<N$, $1<q<q_-$, and $b=N$.
For every $a>0$, $\A_{a,N}$ is attained by a nonnegative
Wulff-symmetric function.
\end{theorem}

\Needspace{8\baselineskip}
At the borderline Taylor order, a finite threshold separates attainment
from nonattainment, and the threshold itself belongs to the attainment range.
\begin{theorem}\label{thm:1.3}
Let $N\ge2$, $0<\beta<N$, $q>1$, $q_-\le q<q_+$, and $b=N$.
There exists a finite threshold $a_c=a_c(N,q,\beta)>N$, independent of
the anisotropy, such that $\A_{a,N}$ is attained if and only if
$0<a\le a_c$. Whenever it is attained, a maximizer may be chosen
nonnegative and Wulff symmetric.
\end{theorem}

\Needspace{18\baselineskip}
Figure~\ref{fig:attainment} summarizes the attainment ranges above.
\begin{figure}[H]
\centering
\begingroup
\small
\setlength{\unitlength}{1mm}
\begin{picture}(158,60)
\linethickness{0.7pt}
\put(0,46){\makebox(0,0)[l]{\parbox[c]{151bp}{\raggedright\setlength{\baselineskip}{17.5bp}$q>1$\\$0<b<N$}}}
\put(60,46){\vector(1,0){87}}
\put(59,46){\circle{2}}
\put(59,41){\makebox(0,0){$0$}}
\put(151,46){\makebox(0,0){$a$}}
\put(103,52){\makebox(0,0){Attained for every $a>0$}}

\put(0,28){\makebox(0,0)[l]{\parbox[c]{151bp}{\raggedright\setlength{\baselineskip}{17.5bp}$1<q<q_-$\\$b=N$}}}
\put(60,28){\vector(1,0){87}}
\put(59,28){\circle{2}}
\put(59,23){\makebox(0,0){$0$}}
\put(151,28){\makebox(0,0){$a$}}
\put(103,34){\makebox(0,0){Attained for every $a>0$}}

\put(0,10){\makebox(0,0)[l]{\parbox[c]{151bp}{\raggedright\setlength{\baselineskip}{17.5bp}$q>1,\quad q_-\le q<q_+$\\$b=N$}}}
\put(60,10){\line(1,0){48}}
\put(59,10){\circle{2}}
\put(108,10){\circle*{2}}
\multiput(110,10)(4,0){9}{\line(1,0){2}}
\put(144,10){\vector(1,0){3}}
\put(59,5){\makebox(0,0){$0$}}
\put(108,5){\makebox(0,0){$a_c$}}
\put(151,10){\makebox(0,0){$a$}}
\put(83,16){\makebox(0,0){Attained}}
\put(128,16){\makebox(0,0){Not attained}}
\end{picture}
\endgroup
\caption{Attainment ranges in Theorems~\ref{thm:1.1}--\ref{thm:1.3}:
solid lines indicate attainment and the dashed line nonattainment,
with the filled point marking attainment at $a_c$.
The middle range may be empty, and positions are not to scale.}
\label{fig:attainment}
\end{figure}

For the natural choice $q=N$, the ranges can be expressed directly in
terms of the singularity exponent.
\begin{corollary}\label{cor:1.2}
Let $N\ge2$, $0<\beta<N$, and $q=N$. Under the assumptions on $F$ above,
the following assertions hold.
\begin{enumerate}
\item If $0<b<N$, then $\A_{a,b}$ is attained for every $a>0$.
\item If $0<\beta\le N/(N-1)$ and $b=N$, there is a finite threshold
$a_c(N,N,\beta)>N$ such that $\A_{a,N}$ is attained if and only if
$0<a\le a_c$.
\item If $N\ge3$, $N/(N-1)<\beta<N$, and $b=N$, then $\A_{a,N}$
is attained for every $a>0$.
\end{enumerate}
In each attainment assertion a maximizer may be chosen nonnegative and
Wulff symmetric. Moreover, for every $a>0$ and $0<b\le N$,
\begin{equation}\label{eq:8}
 \A_{a,b}^F=\frac{\kappa_N}{\omega_N}\A_{a,b}^E,
\end{equation}
where the superscripts $F$ and $E$ refer to the anisotropic and Euclidean
problems, respectively. Attainment
is equivalent in the two problems, and the threshold is independent of $F$.
In dimension two,
\[
 a_c(2,2,\beta)>2+\frac{16}{(2-\beta)^2},\qquad
 a_c(2,2,\beta)\longrightarrow\infty\quad(\beta\uparrow2).
\]
\end{corollary}

The proofs combine the critical--subcritical relation with blow-up
analysis of lifted subcritical maximizers. Exact spatial dilations
reduce attainment to comparison with an endpoint level. The inner
profile is explicit, and the outer limit solves
\begin{equation}\label{eq:9}
 -Q_NG+\eta G^{q-1}=\delta_0\qquad\text{in }\R^N,
\end{equation}
where $Q_Nu:=\diver(F(\nabla u)^{N-1}F_\xi(\nabla u))$.
A Pohozaev identity fixes the optimal Green normalization.
The first retained Taylor term gives strict comparison for $m\le N-2$;
at $m=N-1$, the inner mass correction and a growing-core estimate prove
finiteness of the threshold. The next Taylor term gives threshold
attainment for $N\ge3$, while spatial optimization, outer variations
and a flux barrier handle $N=2$.

Section~\ref{sec:variational} develops the variational reduction and
proves Theorem~\ref{thm:1.1}. Section~\ref{sec:concentration} derives the
concentration bound. Section~\ref{sec:endpoint} proves the endpoint
correction, strict threshold comparison and remaining main results.
Appendix~\ref{sec:A} gives the refined Green-test normalization.

\section{Variational reduction and threshold structure}\label{sec:variational}
\label{sec:2}
As in \cite{ref9}, $u^{\star}$ and $u^{\diamond}$ denote the convex
and Schwarz symmetrizations, respectively. The following reduction
determines the dependence of the supremum on $F$.
\begin{proposition}\label{prop:2.1}
Set $\gamma=(\kappa_N/\omega_N)^{1/N}$ and $c=\gamma^{q/N-1}$.
For $0<b\le N$,
\begin{equation}\label{eq:10}
 \A_{a,b}^F=\gamma^{\beta+q\delta}\A_{a,b}^E,
\end{equation}
where the Euclidean problem has the same constraint and Taylor remainder,
and $\lambda_N$ is replaced by
$\alpha_N=N^{N/(N-1)}\omega_N^{1/(N-1)}$. Attainment is equivalent in
the two problems. In particular, a threshold defined by attainment or
by a strict endpoint inequality is independent of $F$.
\end{proposition}
\begin{proof}
Passing to $u^{\star}$ and using the weighted Hardy--Littlewood inequality
reduce the anisotropic problem to $u(x)=U(F^{o}(x))$, with $U$ nonnegative and
nonincreasing. Set $v(y)=\gamma U(c|y|)$. Polar integration gives
\[
 \norm{\nabla v}_N^N=\norm{F(\nabla u)}_N^N,\qquad
 \norm{v}_q^q=\gamma^{q-N}c^{-N}\norm{u}_q^q=\norm{u}_q^q.
\]
Since $\alpha_N\gamma^{\frac{N}{N-1}}=\lambda_N$, a change of radial variable gives
\[
 \J_F(u)=\gamma^N c^{N-\beta}\J_E(v)=\gamma^{\beta+q\delta}\J_E(v).
\]
The inverse radial map and passage to $u^{\diamond}$ on the Euclidean side
prove the claim. The multiplicative factor is positive and independent of $a$.
\end{proof}

\begin{lemma}\label{lem:weighted}
For $p>q\delta$ and $u\in\D$,
\begin{equation}\label{eq:weighted}
 \int_{\R^N}\frac{|u|^p}{F^{o}(x)^\beta}\dd x
 \le C\norm{F(\nabla u)}_N^{\,p-q\delta}\norm{u}_q^{q\delta},
\end{equation}
where $C=C(N,q,\beta,p,F)$. In particular, for $p>\max\{1,q\delta\}$,
$D^{N,q}$ embeds continuously into $L^p(F^{o}(x)^{-\beta}\dd x)$.
\end{lemma}
\begin{proof}
For $u\ne0$, apply \cite[Theorem 1.1]{ref8} with a fixed strictly subcritical
exponential parameter to $w=u/\norm{F(\nabla u)}_N$ and discard the exponential
factor. Rescaling gives \eqref{eq:weighted}; zero gradient implies $u=0$.
\end{proof}

\begin{lemma}\label{lem:differentiability}
The functional $\J$ is continuously Fr\'echet differentiable on
$\{u\in\D:\norm{F(\nabla u)}_N<1\}$, with
\[
 D\J(u)[h]=\frac{N\lambda_\beta}{N-1}\int_{\R^N}
 \Ph'(\lambda_\beta|u|^{\frac{N}{N-1}})|u|^{\frac{2-N}{N-1}}u\,h\,F^{o}(x)^{-\beta}\dd x.
\]
\end{lemma}
\begin{proof}
Fix $u_0$ with $\norm{F(\nabla u_0)}_N<1$ and choose
$\theta\in(\norm{F(\nabla u_0)}_N,1)$. In a small $D^{N,q}$ neighborhood,
the gradient norm is at most $\theta$ and the $L^q$ norm is bounded.
Put $\dd\mu=F^o(x)^{-\beta}\dd x$ and $r=mN/(N-1)>\max\{1,q\delta\}$.
By Lemma~\ref{lem:weighted}, $D^{N,q}$ embeds into $L^r(\dd\mu)$.
Let
\[
 d(s)=\frac{N\lambda_\beta}{N-1}
 \Ph'(\lambda_\beta|s|^{\frac{N}{N-1}})|s|^{\frac{2-N}{N-1}}s
 \quad(s\ne0),\qquad d(0)=0.
\]
Choose an even smooth cutoff $\chi$, zero on $[-1,1]$ and one for
$|s|\ge2$, and set $d_0=(1-\chi)d$, $d_1=\chi d$.
The bound $|d_0(s)|\le C|s|^{r-1}$ makes the Nemytskii map
$u\mapsto d_0(u)$ continuous from $L^r(\dd\mu)$ to $L^{r'}(\dd\mu)$.

Choose $1<t<t_1$ with $t_1\theta^{N/(N-1)}<1$ and $t'>q\delta$.
The strict exponential margin absorbs the polynomial factors in
$|d_1(u)|^{t_1}$, so the subcritical estimate gives a uniform
$L^{t_1}(\dd\mu)$ bound. If $u_n\to u$ in the chosen neighborhood,
then $u_n\to u$ in $L^\zeta(\dd\mu)$ for
$\zeta>\max\{1,q\delta\}$. For
$E_{n,A}=(\R^N\setminus W_A)\cap\{|u_n|>1\}$,
\[
 \mu(E_{n,A})\le\int_{\R^N\setminus W_A}|u_n|^\zeta\dd\mu,
 \qquad \lim_{A\to\infty}\limsup_{n\to\infty}\mu(E_{n,A})=0.
\]
Since $d_1(u_n)=0$ where $|u_n|\le1$, H\"older's inequality yields
\[
 \int_{\R^N\setminus W_A}|d_1(u_n)|^t\dd\mu
 \le\norm{d_1(u_n)}_{L^{t_1}(\dd\mu)}^t
 \mu(E_{n,A})^{1-t/t_1}.
\]
The same tail bound holds for $u$. On each $W_A$, Vitali's theorem
gives $d_1(u_n)\to d_1(u)$ in $L^t$, hence also on the whole space.
Pairing $d_0$ with $L^r$ variations and $d_1$ with $L^{t'}$ variations
proves continuity in $(D^{N,q})^*$. The fundamental theorem of calculus
then gives the asserted Fr\'echet derivative and $C^1$ regularity.
\end{proof}

For $0<\lambda<\lambda_N$, let $f(\lambda)$ be the subcritical
supremum studied in \cite{ref8,ref9}:
\begin{equation}\label{eq:11}
\begin{aligned}
 f(\lambda)&:=\sup_{\substack{u\in\D\setminus\{0\}\\\norm{F(\nabla u)}_N\le1}}
 \frac1{\norm{u}_q^{q\delta}}\int_{\R^N}
 \frac{\Ph(\lambda\delta|u|^{\frac{N}{N-1}})}{F^{o}(x)^\beta}\dd x.
\end{aligned}
\end{equation}
The sharp critical--subcritical relation of Guo and Liu \cite[Theorem~1.3]{ref8} reads
\begin{equation}\label{eq:12}
 \A_{a,b}=\sup_{0<\lambda<\lambda_N}H_{a,b}(\lambda),
\end{equation}
where
\begin{equation}\label{eq:13}
 H_{a,b}(\lambda):=\left(\frac{1-(\lambda/\lambda_N)^{a(N-1)/N}}
 {(\lambda/\lambda_N)^{b(N-1)/N}}\right)^{q\delta/b}f(\lambda).
\end{equation}
Moreover,
\begin{equation}\label{eq:14}
 f(\lambda)\simeq\left[1-\left(\frac{\lambda}{\lambda_N}\right)^{N-1}\right]^{-q\delta/N}
 \qquad\text{as }\lambda\uparrow\lambda_N.
\end{equation}
The subcritical supremum $f(\lambda)$ has a nonnegative Wulff-symmetric
maximizer for every $\lambda\in(0,\lambda_N)$ by \cite[Theorem~1.1]{ref9},
and is continuous there by \cite[Lemma~4.1]{ref9}.
We use the normalization
\begin{equation}\label{eq:15}
 \norm{F(\nabla u_\lambda)}_N=\norm{u_\lambda}_q=1.
\end{equation}
\begin{lemma}\label{lem:2.2}
For every $a,b>0$, $H_{a,b}(\lambda)\to0$ as $\lambda\downarrow0$. If
\[
 \A_{a,b}>
 \limsup_{\lambda\uparrow\lambda_N}H_{a,b}(\lambda),
\]
then the supremum is attained by a nonnegative Wulff-symmetric function.
\end{lemma}
\begin{proof}
For fixed $\lambda_0\in(0,\lambda_N)$, the first retained Taylor index
and \cite[Theorem 1.1]{ref8} give $f(\lambda)\le C\lambda^m$ for
$0<\lambda\le\lambda_0$. Thus
$H_{a,b}(\lambda)\le C\lambda^{m-q\delta(N-1)/N}\to0$
by \eqref{eq:4}. The strict endpoint inequality and continuity imply
an interior maximum at $\lambda_*$. For a subcritical maximizer $u_*$
satisfying \eqref{eq:15}, set
\[
 \theta_*=(\lambda_*/\lambda_N)^{(N-1)/N},\qquad
 s_*=(1-\theta_*^a)^{1/b},\qquad
 v_*(x)=\theta_*u_*\bigl((\theta_*/s_*)^{q/N}x\bigr).
\]
Dilation gives $\norm{F(\nabla v_*)}_N=\theta_*$,
$\norm{v_*}_q=s_*$ and $\J(v_*)=H_{a,b}(\lambda_*)=\A_{a,b}$.
Hence $v_*$ is admissible, nonnegative and Wulff symmetric.
\end{proof}

\begin{proof}[\textbf{\upshape Proof of Theorem~\ref{thm:1.1}}]
For every $q>1$, \eqref{eq:14} gives
\[
 H_{a,b}(\lambda)\simeq
 \left(1-\frac{\lambda}{\lambda_N}\right)^{q\delta(1/b-1/N)}
 \longrightarrow0\qquad(\lambda\uparrow\lambda_N).
\]
Since the supremum is positive, Lemma~\ref{lem:2.2} applies.
\end{proof}
\label{sec:3}
Henceforth $b=N$; write $H_a=H_{a,N}$ and set
\begin{equation}\label{eq:18}
 \sigma:=\frac{q\delta}{N}=\frac{q(N-\beta)}{N^2}.
\end{equation}
With $r=(\lambda/\lambda_N)^{(N-1)/N}$, we have
\begin{equation}\label{eq:19}
 H_a(\lambda)=\left(\frac{1-r^a}{r^N}\right)^\sigma f(\lambda).
\end{equation}
Define the upper-endpoint level
\begin{equation}\label{eq:20}
 d(a):=\limsup_{\lambda\uparrow\lambda_N}H_a(\lambda).
\end{equation}
By \eqref{eq:14}, $0<d(a)<\infty$.
\begin{lemma}\label{lem:3.1}
For every $a>0$,
\begin{equation}\label{eq:21}
 d(a)=\left(\frac{a}{N}\right)^\sigma d(N).
\end{equation}
\end{lemma}
\begin{proof}
By \eqref{eq:19},
\[
 \frac{H_a(\lambda)}{H_N(\lambda)}
 =\left(\frac{1-r^a}{1-r^N}\right)^\sigma
 \longrightarrow\left(\frac aN\right)^\sigma
 \qquad(\lambda\uparrow\lambda_N).
\]
Taking the limsup gives \eqref{eq:21}.
\end{proof}
For $a>0$ let
\[
 \Sa_a:=\{u\in\D\setminus\{0\}:\norm{F(\nabla u)}_N^a+\norm{u}_q^N=1\}.
\]
Positivity of the Taylor coefficients allows every nonzero interior
competitor to be increased to $\Sa_a$, so $\A_{a,N}=\sup_{\Sa_a}\J$.

\begin{lemma}\label{lem:3.3}
Let $a,a_1>0$, $u\in\Sa_a$, $r=\norm{F(\nabla u)}_N\in(0,1)$, and
\begin{equation}\label{eq:22}
 \tau_{a,{a_1}}(r):=\left(\frac{1-r^a}{1-r^{a_1}}\right)^{q/N^2}.
\end{equation}
Set $(T_{a,{a_1}}u)(x):=u(\tau_{a,{a_1}}(r)x)$. Then $T_{a,{a_1}}:\Sa_a\to\Sa_{a_1}$ is a bijection with inverse
$T_{{a_1},a}$, and
\begin{equation}\label{eq:23}
 \J(T_{a,{a_1}}u)=\left(\frac{1-r^{a_1}}{1-r^a}\right)^\sigma\J(u).
\end{equation}
\end{lemma}
\begin{proof}
Dilation preserves the gradient norm, while
\[
 \norm{T_{a,a_1}u}_q^N=\tau_{a,a_1}^{-N^2/q}(1-r^a)=1-r^{a_1},
 \qquad \J(u(\tau\,\cdot))=\tau^{-(N-\beta)}\J(u).
\]
Thus $T_{a,a_1}u\in\Sa_{a_1}$ and \eqref{eq:23} follows from
\eqref{eq:22}. Interchanging $a$ and $a_1$ gives the inverse.
\end{proof}
\begin{proposition}\label{prop:3.4}
If $0<a<{a_1}$, then
\begin{equation}\label{eq:24}
 \A_{a,N}\le\A_{{a_1},N}\le\left(\frac{{a_1}}{a}\right)^\sigma\A_{a,N}.
\end{equation}
Consequently, $a\mapsto\A_{a,N}$ is continuous and nondecreasing,
while $a\mapsto\A_{a,N}/a^\sigma$ is nonincreasing on $(0,\infty)$. If, in addition, $\A_{{a_1},N}$ is attained, then
\begin{equation}\label{eq:26}
 \frac{\A_{a,N}}{a^\sigma}>\frac{\A_{{a_1},N}}{{a_1}^\sigma}.
\end{equation}
\end{proposition}
\begin{proof}
For $0<r<1$, set $L=-\log r>0$ and $h(t)=(1-r^t)/t$.
Since $h'(t)=[\e^{-Lt}(1+Lt)-1]/t^2<0$, for $0<a<a_1$ we have
\begin{equation}\label{eq:27}
 1<\frac{1-r^{a_1}}{1-r^a}<\frac{{a_1}}{a}.
\end{equation}
Applying Lemma~\ref{lem:3.3} to arbitrary competitors gives \eqref{eq:24}.
Continuity follows from the same two-sided estimate.

If $u_{a_1}\in\Sa_{a_1}$ is a maximizer, then $r_{a_1}:=\norm{F(\nabla u_{a_1})}_N$
lies in $(0,1)$. Applying the inverse map $T_{{a_1},a}$ and using
the strict part of \eqref{eq:27} gives
\[
 \A_{a,N}\ge\J(T_{{a_1},a}u_{a_1})>
 \left(\frac{a}{{a_1}}\right)^\sigma\A_{{a_1},N},
\]
which is \eqref{eq:26}.
\end{proof}
\begin{proposition}\label{prop:3.5}
Assume
\begin{equation}\label{eq:28}
 \A_{N,N}>d(N).
\end{equation}
Define
\begin{equation}\label{eq:29}
 a_c:=\sup\{a>0:\A_{a,N}>d(a)\}.
\end{equation}
Then $a_c\in(N,\infty]$, and
\begin{equation}\label{eq:30}
 \A_{a,N}>d(a),\qquad 0<a<a_c.
\end{equation}
If $a_c<\infty$, then
\begin{equation}\label{eq:31}
 \A_{a,N}=d(a),\qquad a\ge a_c,
\end{equation}
and $\A_{a,N}$ is not attained for $a>a_c$.
\end{proposition}
\begin{proof}
By Lemma~\ref{lem:3.1}, $d(a)/a^\sigma$ is constant, while
Proposition~\ref{prop:3.4} shows that $\A_{a,N}/a^\sigma$ is
nonincreasing. Thus $\{a>0:\A_{a,N}>d(a)\}$ is an initial interval.
Assumption \eqref{eq:28} and continuity give $a_c>N$ and \eqref{eq:30}.
If $a_c<\infty$, then $\A_{a,N}=d(a)$ for $a>a_c$; continuity extends
this equality to $a=a_c$, proving \eqref{eq:31}.

Suppose that $a_0>a_c$ and that $\A_{a_0,N}$ is attained.
Then \eqref{eq:31} gives $\A_{a_0,N}=d(a_0)$. For every $0<{a_1}<a_0$,
the strict inequality \eqref{eq:26} yields
\[
 \A_{{a_1},N}>\left(\frac{{a_1}}{a_0}\right)^\sigma\A_{a_0,N}
 =\left(\frac{{a_1}}{a_0}\right)^\sigma d(a_0)=d({a_1}),
\]
where the last identity follows from Lemma~\ref{lem:3.1}. Hence $a_c\ge a_0$,
a contradiction.
\end{proof}
By Lemma~\ref{lem:2.2}, $\A_{a,N}>d(a)$ implies attainment.
It therefore remains to prove \eqref{eq:28} for $m=N-1$ and the stronger
inequality $\A_{a,N}>d(a)$ for every $a>0$ when $m\le N-2$.

\section{Blow-up analysis and concentration bound}\label{sec:concentration}
\subsection{Blow-up analysis at \texorpdfstring{$a=N$}{a=N}}\label{sec:4}
We now fix $a=b=N$. Choose a sequence
\begin{equation}\label{eq:32}
 \lambda_k\uparrow\lambda_N,\qquad H_N(\lambda_k)\longrightarrow d(N).
\end{equation}
Let $u_k$ be a nonnegative Wulff-symmetric maximizer of $f(\lambda_k)$ satisfying
\eqref{eq:15}. Put
\begin{equation}\label{eq:33}
 \theta_k:=\left(\frac{\lambda_k}{\lambda_N}\right)^{(N-1)/N},\qquad
 s_k:=(1-\theta_k^N)^{1/N},\qquad
 \rho_k:=\left(\frac{\theta_k}{s_k}\right)^{q/N},
\end{equation}
and
\begin{equation}\label{eq:34}
 v_k(x):=\theta_k u_k(\rho_k x).
\end{equation}
Then $\norm{F(\nabla v_k)}_N^N+\norm{v_k}_q^N=1$. Set
\[
E_k:=\int_{\R^N}F(\nabla v_k)^N\dd x=\theta_k^N\to1,\qquad
 S_k:=\int_{\R^N}v_k^q\dd x=s_k^q\to0,
\]
and
\begin{equation}\label{eq:37}
 A_k:=\J(v_k)=H_N(\lambda_k)\longrightarrow d(N)\in(0,\infty).
\end{equation}
\begin{lemma}\label{lem:4.1}
The function $v_k$ maximizes $\J(v)/\norm{v}_q^{q\delta}$ over $v\in\D\setminus\{0\}$ under
$\norm{F(\nabla v)}_N^N\le E_k$. In particular, it maximizes $\J(v)$ under the two constraints
\[
 \int_{\R^N}F(\nabla v)^N\dd x=E_k,\qquad
 \int_{\R^N}|v|^q\dd x=S_k.
\]
\end{lemma}
\begin{proof}
For $w\ne0$ satisfying the gradient constraint, put
$z(x)=\theta_k^{-1}w(x/\rho_k)$. Then $\norm{F(\nabla z)}_N\le1$ and
\[
 \frac{\J(w)}{\norm{w}_q^{q\delta}}
 =\frac{\theta_k^{-q\delta}}{\norm{z}_q^{q\delta}}
 \int_{\R^N}\frac{\Ph(\lambda_k\delta|z|^{\frac{N}{N-1}})}{F^{o}(x)^\beta}\dd x
 \le\theta_k^{-q\delta}f(\lambda_k).
\]
Equality holds for $w=v_k$. Under the two constraints, the denominator is fixed.
\end{proof}
Write $d_k=v_k(0)=\norm{v_k}_\infty$; its finiteness for each fixed $k$
is proved below.
\begin{lemma}\label{lem:4.2}
There exists $\Lambda_k>0$ such that
\begin{equation}\label{eq:38}
 -\Lambda_k Q_Nv_k+\frac{q\delta A_k}{S_k}v_k^{q-1}
 =\frac{N\lambda_\beta}{N-1}\,\Ph'(\lambda_\beta v_k^{\frac{N}{N-1}})v_k^{\frac{1}{N-1}}F^{o}(x)^{-\beta}
\end{equation}
in the weak sense. If
\[
J_k:=\int_{\R^N}\lambda_\beta v_k^{\frac{N}{N-1}}\Ph'(\lambda_\beta v_k^{\frac{N}{N-1}})
 F^{o}(x)^{-\beta}\dd x,
\]
then
\begin{equation}\label{eq:40}
 \Lambda_k E_k=\frac{N}{N-1}J_k-q\delta A_k.
\end{equation}
Moreover,
\begin{equation}\label{eq:41}
 c\le\Lambda_k\le C(1+d_k^{\frac{N}{N-1}}).
\end{equation}
\end{lemma}
\begin{proof}
Since $E_k<1$, Lemma~\ref{lem:differentiability} applies near $v_k$.
Write
\[
 E(v)=\int_{\R^N}F(\nabla v)^N\dd x,\qquad
 S(v)=\int_{\R^N}|v|^q\dd x.
\]
Differentiating the quotient in Lemma~\ref{lem:4.1} under its Dirichlet
constraint gives
\[
 D\J(v_k)=\alpha_kDE(v_k)+\frac{\delta A_k}{S_k}DS(v_k).
\]
The constraint derivative is nonzero since $DE(v_k)[v_k]=NE_k>0$.
Taking $\Lambda_k=N\alpha_k$ yields \eqref{eq:38}. Testing with $v_k$
gives \eqref{eq:40}. Since
\begin{equation}\label{eq:42}
 s\Ph'(s)\ge m\Ph(s),\qquad s\Ph'(s)\le(s+m)\Ph(s),
\end{equation}
we obtain $\Lambda_kE_k\ge(\frac{mN}{N-1}-q\delta)A_k>0$, uniformly bounded away from zero.

We next prove that $d_k<\infty$ for each fixed $k$. Choose $s>1$ close enough
to one that
\[
 s\left(\delta E_k^{1/(N-1)}+\frac{\beta}{N}\right)<1.
\]
The inequality remains strict after a sufficiently small increase in the
exponential coefficient. To use the bounded-domain inequality, write the
radial profile on $W_R$ as $(v_k-V_k(R))_++V_k(R)$; the boundary value is
finite by the radial $L^q$ bound. The estimate
$(x+y)^{\frac{N}{N-1}}\le(1+\varepsilon)x^{\frac{N}{N-1}}+C_\varepsilon y^{\frac{N}{N-1}}$, with $\varepsilon>0$
sufficiently small, preserves the strict inequality above. The singular
weight for the $L^s$ estimate is $F^{o}(x)^{-\beta s}$. The remaining
exponential margin absorbs every fixed polynomial factor, so the nonlinear
density in \eqref{eq:38} belongs to $L^s(W_R)$. The absorption term is
also locally in $L^s$, since local $W^{1,N}$ bounds imply all finite
Lebesgue exponents. Local $W^{1,N}$ membership follows from the Dirichlet
bound and the $L^q$ bound by Poincar\'e's inequality. The radial weak
equation consequently gives
\[
 r^{N-1}|V_k'(r)|^{N-1}\le C_k r^{N(1-1/s)}.
\]
The flux at zero must vanish, since a nonzero limit would imply infinite
$N$-Dirichlet energy. The resulting derivative bound is integrable at zero
because $s>1$. Thus $V_k(0)$ is finite, $V_k$ is continuous at zero, and
its radial flux vanishes there. The integrated radial equation yields
continuity of the radial derivative on $(0,\infty)$.

Finally, the second inequality in \eqref{eq:42} gives
$J_k\le(\lambda_\beta d_k^{\frac{N}{N-1}}+m)A_k$. Together with \eqref{eq:40},
$A_k=O(1)$ and $E_k\to1$, this proves the upper bound in \eqref{eq:41}.
\end{proof}
\begin{lemma}\label{lem:4.3}
One has $d_k\to\infty$ and
\begin{equation}\label{eq:45}
 S_kd_k^{q/(N-1)}\ge c>0
\end{equation}
for all sufficiently large $k$.
\end{lemma}
\begin{proof}
The exact-growth inequality \cite[Theorem~1.4]{ref8}, with $p=q$ and $d=1$, gives
\begin{equation}\label{eq:44}
 \int_{\R^N}\frac{\Ph(\lambda_\beta v_k^{\frac{N}{N-1}})}
 {1+v_k^{q\delta/(N-1)}}F^{o}(x)^{-\beta}\dd x\le CS_k^\delta.
\end{equation}
Since $0\le v_k\le d_k$,
\[
 0<c\le A_k\le CS_k^\delta
 \bigl(1+d_k^{q\delta/(N-1)}\bigr).
\]
Now $S_k\to0$, so \eqref{eq:45} follows, and it implies $d_k\to\infty$.
\end{proof}
\begin{lemma}\label{lem:4.4}
As Radon measures,
\begin{equation}\label{eq:43}
 F(\nabla v_k)^N\dd x\rightharpoonup\delta_0.
\end{equation}
\end{lemma}
\begin{proof}
Write $v_k(x)=V_k(r)$ with $r=F^{o}(x)$. For every fixed $R>0$, radial
monotonicity gives
\[
 S_k\ge\int_{W_R}v_k^q\dd x\ge\kappa_N R^N V_k(R)^q,
\]
whence $V_k(R)\to0$. Suppose that \eqref{eq:43} fails. Then, after passing
to a subsequence, there exist $R>0$ and $\eta\in(0,1)$ such that
\[
 \int_{W_R}F(\nabla v_k)^N\dd x\le1-\eta.
\]
Set $b_k:=V_k(R)$ and $w_k:=(v_k-b_k)_+\in W_0^{1,N}(W_R)$.
Since $\int_{W_R}F(\nabla w_k)^N\dd x\le1-\eta$, choose
$\lambda'>\lambda_\beta$ such that
$\lambda'(1-\eta)^{1/(N-1)}<\lambda_\beta$.
The order-$m$ vanishing of $\Ph$ at zero and the strict exponential
margin $\lambda'-\lambda_\beta>0$ give
\[
 \Ph\bigl(\lambda_\beta(w+b)^{\frac{N}{N-1}}\bigr)
 \le C\Ph\bigl(\lambda'w^{\frac{N}{N-1}}\bigr)
      +Cb^{\frac{mN}{N-1}}
 \qquad(w\ge0,\ 0\le b\le1),
\]
by considering $w\le1$ and $w>1$ separately.
Applying \cite[Corollary~1.2]{ref8} to the zero extension of $w_k$,
rescaled to have gradient norm at most one, and using
$\int_{W_R}F^o(x)^{-\beta}\dd x<\infty$, we obtain
\[
 \int_{W_R}\frac{\Ph(\lambda_\beta v_k^{\frac{N}{N-1}})}{F^{o}(x)^\beta}\dd x
 \le C\norm{w_k}_q^{q\delta}+Cb_k^{\frac{mN}{N-1}}
 \le CS_k^\delta+Cb_k^{\frac{mN}{N-1}}\longrightarrow0.
\]
On $\R^N\setminus W_R$ one has $0\le v_k\le V_k(R)\to0$ uniformly.
Since the first nonzero Taylor term of $\Ph$ has degree $m$ and $\frac{mN}{N-1}>q\delta$,
Lemma~\ref{lem:weighted} gives
\[
 \int_{\R^N\setminus W_R}\frac{\Ph(\lambda_\beta v_k^{\frac{N}{N-1}})}{F^{o}(x)^\beta}\dd x
 \le C\int_{\R^N}\frac{v_k^{\frac{mN}{N-1}}}{F^{o}(x)^\beta}\dd x
 \le CS_k^\delta\longrightarrow0.
\]
Thus $A_k=\J(v_k)\to0$, contradicting \eqref{eq:37}. Therefore the Dirichlet
energy measures converge to the unit point mass at the origin.
\end{proof}

\subsection{Inner blow-up profile and level-truncation estimate}\label{sec:5}
Define the concentration scale $\ell_k$ by
\begin{equation}\label{eq:46}
 \ell_k^{N-\beta}:=\frac{(N-1)\Lambda_k}{N\lambda_\beta}
 d_k^{-\frac{N}{N-1}}\e^{-\lambda_\beta d_k^{\frac{N}{N-1}}},
\end{equation}
and set
\[
\psi_k(x):=\frac{v_k(\ell_k x)}{d_k},\qquad
 \phi_k(x):=d_k^{1/(N-1)}\bigl(v_k(\ell_k x)-d_k\bigr).
\]
\begin{lemma}\label{lem:5.1}
The inner absorption coefficient
$M_k:=\frac{q\delta A_k}{S_k\Lambda_k}d_k^q\ell_k^N$ tends to zero.
\end{lemma}
\begin{proof}
Equations \eqref{eq:46}, \eqref{eq:45}, \eqref{eq:37} and \eqref{eq:41}
give $M_k\le Cd_k^\alpha
\exp[-N\lambda_\beta d_k^{N/(N-1)}/(N-\beta)]\to0$
for some fixed $C,\alpha>0$.
\end{proof}
\begin{proposition}\label{prop:5.2}
Up to a subsequence, $\psi_k\to1$ and $\phi_k\to\phi$ locally,
with $C^1$ convergence away from the origin, where
\begin{equation}\label{eq:48}
 -Q_N\phi=F^{o}(x)^{-\beta}\e^{\frac{N\lambda_\beta}{N-1}\phi}\qquad\text{in }\R^N,
\end{equation}
$\phi\le0$, $\phi(0)=0$, and
\begin{equation}\label{eq:49}
 \phi(x)=-\frac{N-1}{\lambda_\beta}\log\left[
 1+\left(\frac{N\kappa_N}{N-\beta}\right)^{1/(N-1)}
 F^{o}(x)^{\frac{N-\beta}{N-1}}\right].
\end{equation}
Moreover,
\begin{equation}\label{eq:50}
 \int_{\R^N}F^{o}(x)^{-\beta}\e^{\frac{N\lambda_\beta}{N-1}\phi(x)}\dd x=1.
\end{equation}
\end{proposition}
\begin{proof}
Write $r=F^{o}(x)$ and regard $\psi_k$ and $\phi_k$ as radial profiles.
The rescaled equation is
\begin{equation}\label{eq:51}
 -Q_N\phi_k+M_k\psi_k^{q-1}
 =r^{-\beta}\e^{-\lambda_\beta d_k^{\frac{N}{N-1}}}
 \Ph'(\lambda_\beta d_k^{\frac{N}{N-1}}\psi_k^{\frac{N}{N-1}})\psi_k^{\frac{1}{N-1}}.
\end{equation}
Because $0\le\psi_k\le1$ and $\Ph'(s)\le\e^s$, its right-hand side is at
most $r^{-\beta}$. The profiles are nonincreasing and have zero flux at
the origin. Integrating the radial equation and dropping the nonnegative
absorption term yields
\[
 0\le r^{N-1}(-\phi_k'(r))^{N-1}\le\frac{r^{N-\beta}}{N-\beta}.
\]
Consequently, with $\gamma=(N-\beta)/(N-1)>0$,
\[
 0\le-\phi_k(r)\le Cr^\gamma,\qquad|\phi_k'(r)|\le Cr^{\gamma-1}.
\]
These estimates give uniform convergence along a subsequence on every
bounded radial interval, including the origin. Since $\psi_k=1+d_k^{-\frac{N}{N-1}}\phi_k$,
we also have $\psi_k\to1$ locally uniformly. The right-hand side of
\eqref{eq:51} converges to $r^{-\beta}\e^{\frac{N\lambda_\beta}{N-1}\phi}$ and is
dominated by $r^{-\beta}$; the absorption term tends to zero by
Lemma~\ref{lem:5.1}. Passing to the integrated equation gives
\begin{equation}\label{eq:52}
 r^{N-1}(-\phi'(r))^{N-1}
 =\int_0^r s^{N-1-\beta}\e^{\frac{N\lambda_\beta}{N-1}\phi(s)}\dd s,\qquad\phi(0)=0.
\end{equation}
Convergence of the fluxes is uniform on compact intervals, and hence
derivative convergence is uniform away from zero.

The initial-value problem \eqref{eq:52} has at most one locally bounded
nonincreasing solution. Indeed, near zero the integrals on its right-hand
side are bounded above and below by positive multiples of $r^{N-\beta}$.
For two solutions, the mean-value estimate for the $(N-1)$st root therefore gives
\[
 |\phi_1'(r)-\phi_2'(r)|
 \le Cr^{\gamma-1}\sup_{0\le s\le r}|\phi_1(s)-\phi_2(s)|.
\]
Integration gives equality on a sufficiently short interval. Uniqueness
then extends away from zero by Gronwall's inequality applied to the same
integral equation. Direct substitution shows that \eqref{eq:49} solves
\eqref{eq:52}. Finally, with $B=(N\kappa_N/(N-\beta))^{1/(N-1)}$,
\[
 \int_{\R^N}F^{o}(x)^{-\beta}\e^{\frac{N\lambda_\beta}{N-1}\phi}\dd x
 =N\kappa_N\int_0^\infty\frac{r^{N-1-\beta}}{(1+Br^\gamma)^N}\dd r
 =\frac{N\kappa_N B^{-(N-1)}}{N-\beta}=1.
\]
This proves both the profile and the mass assertion.
\end{proof}
\Needspace{8\baselineskip}
The next estimate controls the Dirichlet energy of the level truncations
used in the multiplier and source-term limits.
\begin{lemma}\label{lem:5.3}
For every $L>1$,
\begin{equation}\label{eq:53}
 \limsup_{k\to\infty}\int_{\R^N}
 F\left(\nabla\min\left\{v_k,\frac{d_k}{L}\right\}\right)^N\dd x\le\frac1L.
\end{equation}
\end{lemma}
\begin{proof}
Let $R_{k,L}$ be determined by $V_k(R_{k,L})=d_k/L$. Fix $R_0>0$.
Since $V_k(R_0)\to0$, H\"older's inequality on the annulus
$W_{R_0}\setminus W_{R_{k,L}}$ gives
\[
 \frac{d_k}{L}-V_k(R_0)\le
 C\left(\int_{W_{R_0}\setminus W_{R_{k,L}}}F(\nabla v_k)^N\dd x\right)^{1/N}
 \left(\log\frac{R_0}{R_{k,L}}\right)^{(N-1)/N}.
\]
Since $E_k\le1$, this implies
\begin{equation}\label{eq:54}
 R_{k,L}\le C\e^{-c_Ld_k^{\frac{N}{N-1}}}.
\end{equation}
Because $0\le v_k\le d_k$ and $\{v_k>d_k/L\}=W_{R_{k,L}}$,
\[
 \int_{\{v_k>d_k/L\}}v_k^q\dd x\le Cd_k^q\e^{-c_LNd_k^{\frac{N}{N-1}}}.
\]
Together with \eqref{eq:45}, which gives $S_k\ge cd_k^{-q/(N-1)}$, we obtain
\begin{equation}\label{eq:55}
 \frac1{S_k}\int_{\{v_k>d_k/L\}}v_k^q\dd x\longrightarrow0.
\end{equation}
Set $w_{k,L}:=(v_k-d_k/L)_+$. Testing \eqref{eq:38} with $w_{k,L}$
and dividing by $\Lambda_k$ yields
\begin{align*}
 &\int_{\{v_k>d_k/L\}}F(\nabla v_k)^N\dd x+
 \frac{q\delta A_k}{S_k\Lambda_k}\int_{\{v_k>d_k/L\}}v_k^{q-1}w_{k,L}\dd x\\
 &\qquad=\frac{N\lambda_\beta}{(N-1)\Lambda_k}\int_{\{v_k>d_k/L\}}
 \Ph'(\lambda_\beta v_k^{\frac{N}{N-1}})v_k^{\frac{1}{N-1}}w_{k,L}F^{o}(x)^{-\beta}\dd x.
\end{align*}
The second term on the left tends to zero by \eqref{eq:55} and the lower
bound for $\Lambda_k$. On $W_{R\ell_k}$, one has $v_k/d_k\to1$ locally
after the inner rescaling, and the normalized nonlinear density converges
to $F^{o}(y)^{-\beta}\e^{\frac{N\lambda_\beta}{N-1}\phi(y)}\dd y$. Hence, for every fixed $R>0$,
\[
 \liminf_{k\to\infty}\int_{\R^N}F(\nabla w_{k,L})^N\dd x
 \ge\left(1-\frac1L\right)
 \int_{W_R}F^{o}(y)^{-\beta}\e^{\frac{N\lambda_\beta}{N-1}\phi(y)}\dd y.
\]
Letting $R\to\infty$ and using \eqref{eq:50} gives
\[
 \liminf_{k\to\infty}\int_{\R^N}F(\nabla w_{k,L})^N\dd x\ge1-\frac1L.
\]
Finally, $\nabla w_{k,L}$ and $\nabla\min\{v_k,d_k/L\}$ have disjoint
supports and their $N$-energies add up to $E_k\to1$. This proves \eqref{eq:53}.
\end{proof}
\Needspace{10\baselineskip}
We now use the truncation estimate to determine the leading terms of
$J_k$ and $\Lambda_k$.
\begin{lemma}\label{lem:5.4}
One has
\begin{equation}\label{eq:56}
 \frac{J_k}{\lambda_\beta d_k^{\frac{N}{N-1}} A_k}\longrightarrow1
\end{equation}
and
\begin{equation}\label{eq:57}
 \frac{\Lambda_k}{A_k d_k^{\frac{N}{N-1}}}\longrightarrow \frac{N\lambda_\beta}{N-1}.
\end{equation}
\end{lemma}
\begin{proof}
Fix $L>1$ and decompose $\R^N$ into
\[
\Omega^+_{k,L}:=\{v_k>d_k/L\},\qquad
\Omega^-_{k,L}:=\{v_k\le d_k/L\}.
\]
Set $z_k:=\min\{v_k,d_k/L\}$. By Lemma~\ref{lem:5.3}, its Dirichlet energy is at most $L^{-1}+o(1)$. Choose $\theta\in(L^{-1/N},1)$; then $\|F(\nabla z_k)\|_N\le\theta$ for large $k$. Applying \cite[Corollary~1.2]{ref8} to $z_k/\theta$ at the strictly subcritical parameter $\lambda_N\theta^{\frac{N}{N-1}}$ gives
\[
\int_{\Omega^-_{k,L}}
 \frac{\Ph(\lambda_\beta v_k^{\frac{N}{N-1}})}{F^o(x)^\beta}\,dx
 \le C\|z_k\|_q^{q\delta}\le CS_k^\delta=o(1).
\]
Since $A_k\to d(N)>0$, the integral over $\Omega^+_{k,L}$ equals $A_k+o(1)$. On this set, $\lambda_\beta v_k^{\frac{N}{N-1}}\to\infty$ uniformly as $k\to\infty$, and
\[
\frac{\Phi'_{N,q,\beta}(s)}{\Ph(s)}\longrightarrow1
\qquad(s\to\infty).
\]
Consequently, for fixed $L$,
\[
\liminf_{k\to\infty}
 \frac{J_k}{\lambda_\beta d_k^{\frac{N}{N-1}}A_k}
 \ge L^{-\frac{N}{N-1}}.
\]
On the other hand, \eqref{eq:42} gives
\[
J_k\le\bigl(\lambda_\beta d_k^{\frac{N}{N-1}}+m\bigr)A_k,
\]
so the corresponding limsup is at most one. Letting $L\downarrow1$ proves \eqref{eq:56}. Finally, divide \eqref{eq:40} by $A_kd_k^{\frac{N}{N-1}}$ and use $E_k\to1$ and $d_k\to\infty$ to obtain \eqref{eq:57}.
\end{proof}

\subsection{The nonlinear Green function}\label{sec:6}
Set $T_k:=S_kd_k^{q/(N-1)}$; by \eqref{eq:45}, $T_k\ge c>0$.
\begin{lemma}\label{lem:6.1}
There exists $C>0$ such that $0<c\le T_k\le C$.
\end{lemma}
\begin{proof}
For fixed admissible $u$, the function
\[
 g_u(\lambda):=\frac1{\norm u_q^{q\delta}}
 \int_{\R^N}\Ph(\lambda\delta|u|^{\frac{N}{N-1}})F^o(x)^{-\beta}\dd x
\]
is convex, and $f=\sup_u g_u$. Since
$g_{u_k}(\lambda_k)=f(\lambda_k)$ and $g_{u_k}\le f$, choosing
$s=(\lambda_k+\lambda_N)/2$ and using \eqref{eq:14} gives
\[
 \frac{g_{u_k}'(\lambda_k)}{f(\lambda_k)}
 \le\frac{f(s)/f(\lambda_k)-1}{s-\lambda_k}
 \le\frac{C}{\lambda_N-\lambda_k}.
\]
Differentiation under the integral and the lifting
\eqref{eq:33}--\eqref{eq:34} give
$J_k/A_k=\lambda_k g_{u_k}'(\lambda_k)/f(\lambda_k)$.
Lemma~\ref{lem:5.4} therefore yields
$d_k^{N/(N-1)}\le C/(\lambda_N-\lambda_k)$. Consequently,
\[
 T_k^{N/q}=(1-\theta_k^N)d_k^{N/(N-1)}\le C,
\]
because $1-\theta_k^N\simeq\lambda_N-\lambda_k$.
The lower bound is \eqref{eq:45}.
\end{proof}
Pass to a subsequence with $T_k\to T\in(0,\infty)$ and set
$G_k:=d_k^{1/(N-1)}v_k$, so $\norm{G_k}_q^q=T_k$.
Multiplying \eqref{eq:38} by $d_k/\Lambda_k$ gives
\[
 -Q_NG_k+\eta_kG_k^{q-1}=\nu_k,\qquad
 \eta_k:=\frac{q\delta A_k}{S_k\Lambda_k}d_k^{\frac{N-q}{N-1}},
\]
where
\begin{equation}\label{eq:63}
 \nu_k:=\frac{N\lambda_\beta}{(N-1)\Lambda_k}
 d_k\Ph'(\lambda_\beta v_k^{\frac{N}{N-1}})v_k^{\frac{1}{N-1}}F^{o}(x)^{-\beta}\dd x.
\end{equation}
By \eqref{eq:57}, $\eta_kT_k\to q(N-1)/(N\lambda_N)$, hence
\begin{equation}\label{eq:65}
 \eta_k\to\eta:=\frac{q(N-1)}{N\lambda_N T}>0.
\end{equation}
\begin{lemma}\label{lem:6.2}
The measures $\nu_k$ in \eqref{eq:63} satisfy
\[
 \nu_k\rightharpoonup\delta_0\qquad\text{locally in }\mathcal{M}(\R^N).
\]
\end{lemma}
\begin{proof}
Put
\[
 \dd\mu_k=\frac{N\lambda_\beta}{(N-1)\Lambda_k}
 v_k^{\frac{N}{N-1}}\Ph'(\lambda_\beta v_k^{\frac{N}{N-1}})F^{o}(x)^{-\beta}\dd x.
\]
By Lemma~\ref{lem:5.4}, $\mu_k(\R^N)=\frac{N J_k}{(N-1)\Lambda_k}\to1$.
The inner change of variables and Proposition~\ref{prop:5.2} give
\[
 \lim_{k\to\infty}\mu_k(W_{R\ell_k})
 =\int_{W_R}F^{o}(y)^{-\beta}\e^{\frac{N\lambda_\beta}{N-1}\phi(y)}\dd y.
\]
Thus $\mu_k\rightharpoonup\delta_0$, and its mass outside the inner balls
tends to zero as first $k\to\infty$ and then $R\to\infty$.

Fix $L>1$ and a bounded Wulff ball $W_B$. On $\{v_k>d_k/L\}$,
$\dd\mu_k\le\dd\nu_k\le L\dd\mu_k$. On its complement set
$z_k=\min\{v_k,d_k/L\}$. Its gradient energy is at most $L^{-1}+o(1)$ by
Lemma~\ref{lem:5.3}. The bounded-domain singular subcritical inequality,
with a small exponential margin, implies
\[
 \sup_k\int_{W_B}\Ph'(\lambda_\beta z_k^{\frac{N}{N-1}})z_k^{\frac{1}{N-1}}F^{o}(x)^{-\beta}\dd x
 <\infty.
\]
Here the bounded additive mean is controlled by the local $L^q$ norm;
the polynomial factor is absorbed by the exponential margin. Since
$d_k/\Lambda_k=O(d_k^{-\frac{1}{N-1}})\to0$, the sublevel contribution to $\nu_k(W_B)$
tends to zero. The inner ball gives the lower bound one for the limiting
mass near zero; the superlevel comparison gives an upper bound $L$, and
no mass remains on a fixed annulus. Letting $L\downarrow1$ proves local
convergence to $\delta_0$.
\end{proof}
\begin{proposition}\label{prop:6.3}
Up to a subsequence,
\[
 G_k\rightharpoonup G\ \text{in }W_{\mathrm{loc}}^{1,r}(\R^N)\ (1<r<N),
 \qquad G_k\to G\ \text{in }C_{\mathrm{loc}}^1(\R^N\setminus\{0\}).
\]
The limit is nonnegative, Wulff radial, satisfies $\norm{G}_q^q\le T$, and solves
\begin{equation}\label{eq:66}
 -Q_NG+\eta G^{q-1}=\delta_0\qquad\text{in }\R^N.
\end{equation}
Writing $G(x)=g(r)$, $r=F^{o}(x)$, its radial flux is
\begin{equation}\label{eq:70}
 P(r):=N\kappa_Nr^{N-1}(-g'(r))^{N-1}
 =1-\eta\int_{W_r}G^{q-1}\dd x\ge0.
\end{equation}
As $r\downarrow0$,
\begin{equation}\label{eq:71}
 1-P(r)=O\bigl(r^N|\log r|^{q-1}\bigr),
\end{equation}
and, for some finite $C_\eta$,
\begin{equation}\label{eq:73}
 G(x)=-(N\kappa_N)^{-1/(N-1)}\log F^{o}(x)+C_\eta
 +O\bigl(F^{o}(x)^N|\log F^{o}(x)|^{q-1}\bigr).
\end{equation}
At infinity, $g(r)\to0$, $P(r)\to0$, and $r^Ng(r)^q\to0$.
\end{proposition}
\begin{proof}
Write $G_k(x)=g_k(r)$. The radial equation gives
\begin{equation}\label{eq:68}
 P_k(r):=N\kappa_Nr^{N-1}(-g_k'(r))^{N-1}
 =\nu_k(W_r)-\eta_k\int_{W_r}G_k^{q-1}\dd x\ge0.
\end{equation}
For fixed $R$, Lemma~\ref{lem:6.2} bounds $\nu_k(W_R)$ uniformly.
Radial monotonicity and $\norm{G_k}_q^q=T_k\le C$ therefore give
\[
 |g_k'(r)|\le C_R/r,\qquad
 0\le g_k(r)\le C_R(1+|\log(r/R)|),\qquad 0<r\le R.
\]
These estimates yield local $W^{1,r}$ bounds for $1<r<N$, subsequential
uniform convergence on annuli, and convergence in every finite local
Lebesgue space. In particular, $G_k^{q-1}\to G^{q-1}$ in $L^1(W_R)$,
also when $q<2$.

On every $[r_0,R]\subset(0,\infty)$, the monotone distribution functions
$\nu_k(W_r)$ converge uniformly to one by Lemma~\ref{lem:6.2}.
Thus \eqref{eq:68} converges uniformly to \eqref{eq:70}.
Taking the $(N-1)$st root gives uniform derivative convergence on annuli.
The logarithmic bound gives $P(r)\to1$ at zero, which identifies the
source in \eqref{eq:66}, and Fatou's lemma gives $\norm{G}_q^q\le T$.

Substitution of the logarithmic bound into \eqref{eq:70} proves
\eqref{eq:71}; integration of
$-g'(r)=(N\kappa_N)^{-1/(N-1)}P(r)^{1/(N-1)}/r$ gives \eqref{eq:73}.
At infinity, monotonicity and $G\in L^q$ imply $g(r)\to0$ and
$r^Ng(r)^q\to0$. The nonnegative, nonincreasing flux must also tend to
zero: a positive limit would force $-g'(r)\ge c/r$, contradicting $g\ge0$.
\end{proof}
For large $t$, let $r(t)$ be determined by $g(r(t))=t$ and set
\begin{equation}\label{eq:72}
 \mathcal{P}(t):=P(r(t)).
\end{equation}

\begin{lemma}\label{lem:6.4}
The nonlinear Green function satisfies
\begin{equation}\label{eq:74}
 \eta\norm{G}_q^q=\frac{q(N-1)}{N\lambda_N}.
\end{equation}
Moreover, $G_k\to G$ strongly in $L^q(\R^N)$ and
\begin{equation}\label{eq:67}
 \norm{G}_q^q=T.
\end{equation}
\end{lemma}
\begin{proof}
Set $r=\e^s$ and $h(s)=g(\e^s)$. Since
\[
 \int_{\R^N}G^q\dd x=N\kappa_N\int_{-\infty}^\infty\e^{Ns}h(s)^q\dd s,
\]
the radial equation associated with \eqref{eq:66} on $\R^N\setminus\{0\}$
can be written as
\[
 -\bigl(|h'|^{N-2}h'\bigr)'+\eta\e^{Ns}h^{q-1}=0.
\]
Multiplying by $h'$ and integrating from $-R$ to $R$ gives
\[
 -\frac{N-1}{N}\left[|h'(s)|^N\right]_{-R}^R+
 \frac{\eta}{q}\left[\e^{Ns}h(s)^q\right]_{-R}^R
 -\frac{N\eta}{q}\int_{-R}^R\e^{Ns}h(s)^q\dd s=0.
\]
The flux limits in Proposition~\ref{prop:6.3} give
\[
 h'(s)\longrightarrow-(N\kappa_N)^{-1/(N-1)}\qquad(s\to-\infty),
\]
whereas $h'(s)\to0$ as $s\to+\infty$. Moreover $\e^{Ns}h(s)^q\to0$ at both
ends; at $+\infty$ this follows from monotonicity and $G\in L^q$.
Letting $R\to\infty$ therefore yields
\[
 \frac{N-1}{N (N\kappa_N)^{N/(N-1)}}=\frac{N\eta}{q}
 \int_{-\infty}^\infty\e^{Ns}h(s)^q\dd s.
\]
Multiplying by $qN\kappa_N/N$ and using $\lambda_N=N(N\kappa_N)^{1/(N-1)}$
gives \eqref{eq:74}. Comparing with \eqref{eq:65} proves \eqref{eq:67}.
Boundedness and local convergence imply weak $L^q$ convergence; since
$\norm{G_k}_q^q=T_k\to T=\norm{G}_q^q$, uniform convexity gives strong convergence.
\end{proof}
Testing \eqref{eq:66} with $G$ on the finite annulus $W_R\setminus W_r$
gives the boundary terms $g(r)P(r)-g(R)P(R)$. Since $g(R)\to0$ and
$P(R)\to0$, the outer boundary term vanishes as $R\to\infty$.
Monotone convergence of the nonnegative energy terms therefore gives
\[
\int_{\R^N\setminus W_r}F(\nabla G)^N\dd x+
 \eta\int_{\R^N\setminus W_r}G^q\dd x=g(r)P(r).
\]
Using \eqref{eq:71} and \eqref{eq:74},
\begin{equation}\label{eq:76}
 \int_{\R^N\setminus W_r}F(\nabla G)^N\dd x
 =G(r)-\frac{q(N-1)}{N\lambda_N}+o(1)\qquad(r\downarrow0).
\end{equation}

\subsection{Concentration bound and Green-function normalization}\label{sec:7}
Along the sequence \eqref{eq:32}, set $S:=T^{N/q}=\norm{G}_q^N$
and $c_q:=q(N-1)/(N\lambda_N)$. The global constraint contributes
$S-c_q$ to the concentration expansion at $a=N$.

Fix $\rho>0$ and write
\[
 b_{k,\rho}:=v_k|_{\partial W_\rho},\qquad
 \bar v_{k,\rho}:=(v_k-b_{k,\rho})_+,\qquad
 \tau_{k,\rho}:=\int_{W_\rho}F(\nabla\bar v_{k,\rho})^N\dd x.
\]
The outer convergence gives
\begin{equation}\label{eq:79}
 d_k^{1/(N-1)}b_{k,\rho}\to G(\rho).
\end{equation}
On each bounded annulus the gradients converge uniformly. Letting the
outer radius tend to infinity gives the required lower bound:
\begin{equation}\label{eq:80}
 \liminf_{k\to\infty}d_k^{\frac{N}{N-1}}\int_{\R^N\setminus W_\rho}F(\nabla v_k)^N\dd x
 \ge\int_{\R^N\setminus W_\rho}F(\nabla G)^N\dd x.
\end{equation}
Since $d_k^{N/(N-1)}(1-E_k)\to S$, we obtain
\begin{equation}\label{eq:82}
\begin{aligned}
 \liminf_{k\to\infty}d_k^{\frac{N}{N-1}}(1-\tau_{k,\rho})
 &\ge S+\int_{\R^N\setminus W_\rho}F(\nabla G)^N\dd x\\
 &=G(\rho)+S-c_q+o_\rho(1).
\end{aligned}
\end{equation}
The lower bound in \eqref{eq:80} suffices for the following estimate.

The elementary inequality
\begin{equation}\label{eq:83}
 (z+b)^{\frac{N}{N-1}}-\tau^{-1/(N-1)}z^{\frac{N}{N-1}}\le\frac{b^{\frac{N}{N-1}}}{(1-\tau)^{1/(N-1)}},
 \qquad z,b\ge0,\quad0<\tau<1,
\end{equation}
combined with the anisotropic singular Carleson--Chang estimate of
\cite{ref10} yields the following bound.
\begin{proposition}\label{prop:7.1}
For the endpoint sequence \eqref{eq:32}, $d(N)\le\mathcal{C}_N(S)$, where
\begin{equation}\label{eq:85}
 \mathcal{C}_N(S):=\frac{N\kappa_N}{N-\beta}
 \exp\left[H_{N-1}+\lambda_\beta C_\eta
 -\frac{\lambda_\beta}{N-1}(S-c_q)\right].
\end{equation}
\end{proposition}
\begin{proof}
Fix $\rho>0$ small enough that
\[
B_\rho:=S+\int_{\R^N\setminus W_\rho}F(\nabla G)^N\,dx>0.
\]
Set $z_{k,\rho}:=\tau_{k,\rho}^{-1/N}\bar v_{k,\rho}$. For large $k$, $0<\tau_{k,\rho}<1$, and Lemma~\ref{lem:4.4} gives $\tau_{k,\rho}\to1$. Thus $z_{k,\rho}\in W_0^{1,N}(W_\rho)$ has unit Dirichlet energy, and
$\|z_{k,\rho}\|_q\le\tau_{k,\rho}^{-1/N}\|v_k\|_q\to0$, so $z_{k,\rho}\rightharpoonup0$ in $W_0^{1,N}(W_\rho)$. For every $0<r<\rho$, its Dirichlet energy on $W_\rho\setminus W_r$ tends to zero by the same lemma. These are the three hypotheses of \cite[Lemma~3.12]{ref10}.

Fix $L>1$ and put $E_{k,L}:=\{v_k>d_k/L\}$. The proof of Lemma~\ref{lem:5.4} shows that its complement contributes $o(1)$ to $A_k$. Also, $E_{k,L}\subset W_\rho$ for large $k$ and its weighted measure tends to zero by \eqref{eq:54}. On $E_{k,L}$, \eqref{eq:83} gives
\[
e^{\lambda_\beta v_k^{\frac{N}{N-1}}}\le e^{B_{k,\rho}}e^{\lambda_\beta z_{k,\rho}^{\frac{N}{N-1}}},
\qquad B_{k,\rho}:=
\frac{\lambda_\beta b_{k,\rho}^{\frac{N}{N-1}}}{(1-\tau_{k,\rho})^{1/(N-1)}}.
\]
By \eqref{eq:79} and \eqref{eq:82},
\[
\limsup_k B_{k,\rho}\le
\lambda_\beta G(\rho)^{\frac{N}{N-1}}/B_\rho^{1/(N-1)}<\infty.
\]
Since $\Ph\le\exp$, subtracting one on the shrinking set $E_{k,L}$ gives
\[
A_k\le e^{B_{k,\rho}}\int_{W_\rho}
\bigl(e^{\lambda_\beta z_{k,\rho}^{\frac{N}{N-1}}}-1\bigr)F^o(x)^{-\beta}\,dx+o(1).
\]
The $o(1)$ includes the weighted measure of $E_{k,L}$ for fixed $\rho$. The anisotropic singular Carleson--Chang estimate \cite[Lemma~3.12]{ref10} now yields
\[
d(N)\le\frac{N\kappa_N}{N-\beta}\rho^{N-\beta}e^{H_{N-1}}
\exp\!\left(\frac{\lambda_\beta G(\rho)^{\frac{N}{N-1}}}{B_\rho^{1/(N-1)}}\right).
\]
By \eqref{eq:76}, $B_\rho=G(\rho)+S-c_q+o_\rho(1)$, so
\[
\frac{G(\rho)^{\frac{N}{N-1}}}{B_\rho^{1/(N-1)}}
=G(\rho)-\frac{S-c_q}{N-1}+o_\rho(1).
\]
Finally $\rho^{N-\beta}e^{\lambda_\beta G(\rho)}\to e^{\lambda_\beta C_\eta}$ by \eqref{eq:73}. Letting $\rho\downarrow0$ gives \eqref{eq:85}.
\end{proof}
We now vary $S$ within the dilation orbit of this Green function;
no uniqueness among all Green solutions is assumed.
\begin{lemma}\label{lem:7.2}
Let $G$ solve \eqref{eq:66} and set $G_\tau(x):=G(\tau x)$ for $\tau>0$.
Then
\[
 -Q_NG_\tau+\tau^N\eta G_\tau^{q-1}=\delta_0,
\]
\begin{equation}\label{eq:86}
 \norm{G_\tau}_q^N=\tau^{-N^2/q}S,
\end{equation}
and, if $C(S)$ denotes the regular constant in the pole expansion, then
\begin{equation}\label{eq:87}
 \lambda_\beta\bigl(C(S_2)-C(S_1)\bigr)=\sigma\log\frac{S_2}{S_1}.
\end{equation}
\end{lemma}
\begin{proof}
Homogeneity gives the equation and \eqref{eq:86}. The pole expansion gives
$C_{\tau^N\eta}=C_\eta-(N\kappa_N)^{-1/(N-1)}\log\tau$.
Since $\lambda_\beta(N\kappa_N)^{-1/(N-1)}=N-\beta$ and
$\log(S_2/S_1)=-(N^2/q)\log\tau$, this is \eqref{eq:87}.
\end{proof}
Consequently, there exists a constant $A_0>0$ such that
\begin{equation}\label{eq:88}
 \mathcal{C}_N(S)=A_0S^\sigma\exp\left(-\frac{\lambda_\beta}{N-1}S\right).
\end{equation}
The function in \eqref{eq:88} has its unique maximum at
\begin{equation}\label{eq:89}
 S=\frac{(N-1)\sigma}{\lambda_\beta}=\frac{q(N-1)}{N\lambda_N}=c_q.
\end{equation}
\Needspace{9\baselineskip}
Let $G_*$ denote the dilation of $G$ satisfying
\begin{equation}\label{eq:90}
 \norm{G_*}_q^N=c_q.
\end{equation}
Then Lemma~\ref{lem:6.4} implies
\begin{equation}\label{eq:91}
 -Q_NG_*+\norm{G_*}_q^{N-q}G_*^{q-1}=\delta_0.
\end{equation}
Write
\begin{equation}\label{eq:92}
 G_*(x)=-(N\kappa_N)^{-1/(N-1)}\log F^{o}(x)+A_*+o(1).
\end{equation}
For this normalization, define
\[
D_N:=\frac{N\kappa_N}{N-\beta}\exp\bigl(H_{N-1}+\lambda_\beta A_*\bigr).
\]
\begin{corollary}\label{cor:7.3}
One has
\begin{equation}\label{eq:94}
 d(N)\le D_N.
\end{equation}
Moreover,
\[
\frac{\mathcal{C}_N(S)}{D_N}
 =\left(\frac{S}{c_q}\right)^\sigma
 \exp\left[-\sigma\left(\frac{S}{c_q}-1\right)\right]\le1,
\]
with equality if and only if $S=c_q$.
\end{corollary}
\begin{proof}
The formula follows from \eqref{eq:88} and \eqref{eq:89}. The inequality
is equivalent to $\log x\le x-1$ for $x>0$.
Combining with Proposition~\ref{prop:7.1} gives \eqref{eq:94}.
\end{proof}

\section{Test functions and proofs of the endpoint results}\label{sec:endpoint}
\subsection{Green-function tests}\label{sec:8}
We construct an admissible function for the $a=N$ constraint whose
functional value is strictly larger than $D_N$. We adapt the Green-function
constructions in \cite{ref10,ref11}. The estimates required in the borderline
case $m=N-1$ are established in Appendix~\ref{sec:A}.

Set
\[
 \gamma:=\frac{N-\beta}{N-1},\quad
 B:=\left(\frac{N\kappa_N}{N-\beta}\right)^{1/(N-1)},\quad
 L_\varepsilon:=\log\frac1\varepsilon,\quad R=L_\varepsilon^M,
\]
where $M>0$ will be chosen large. With
$t_\varepsilon:=\frac{N}{\lambda_N}\log\frac1{R\varepsilon}$ and
$s_\varepsilon(t):=B\varepsilon^{-\gamma}\e^{-\lambda_\beta t/(N-1)}$,
define
\begin{equation}\label{eq:profile}
 h_\varepsilon(t)=
 \begin{cases}
 t_\varepsilon+\dfrac{N-1}{\lambda_\beta}
 \log\dfrac{1+BR^\gamma}{1+s_\varepsilon(t)},&t\ge t_\varepsilon,\\[1ex]
 t,&t<t_\varepsilon.
 \end{cases}
\end{equation}
Since $s_\varepsilon(t_\varepsilon)=BR^\gamma$, the profile is continuous.
Let
\begin{equation}\label{eq:100}
 Y_\varepsilon:=\norm{F(\nabla h_\varepsilon(G_*))}_N^N
 +\norm{h_\varepsilon(G_*)}_q^N,
\end{equation}
and set $\varphi_\varepsilon:=Y_\varepsilon^{-1/N}h_\varepsilon(G_*)$.
The outer branch avoids the pole; the inner branch is bounded with
exponentially decaying derivative. Thus $Y_\varepsilon<\infty$, and
homogeneity gives
\begin{equation}\label{eq:101}
 \norm{F(\nabla\varphi_\varepsilon)}_N^N+\norm{\varphi_\varepsilon}_q^N=1.
\end{equation}

\begin{lemma}\label{lem:8.1}
Choose $M$ so that
\begin{equation}\label{eq:102}
 M\gamma>1.
\end{equation}
The function $\varphi_\varepsilon$ defined above satisfies \eqref{eq:101}, and
\begin{equation}\label{eq:103}
 Y_\varepsilon=\frac{N}{\lambda_N}L_\varepsilon+
 \frac1{\lambda_\beta}\log\frac{N\kappa_N}{N-\beta}-
 \frac{N-1}{\lambda_\beta}H_{N-1}+o(Y_\varepsilon^{-1}).
\end{equation}
Moreover,
\begin{equation}\label{eq:104}
 \int_{\{G_*>t_\varepsilon\}}
 \frac{\Ph(\lambda_\beta\varphi_\varepsilon^{\frac{N}{N-1}})}{F^{o}(x)^\beta}\dd x
 \ge D_N+o(Y_\varepsilon^{-1}).
\end{equation}
On every fixed Wulff annulus $W_{2r_0}\setminus W_{r_0}$,
for sufficiently small $\varepsilon$,
\begin{equation}\label{eq:105}
 \varphi_\varepsilon=Y_\varepsilon^{-1/N}G_*\qquad\text{on }W_{2r_0}\setminus W_{r_0}.
\end{equation}
Finally,
\begin{equation}\label{eq:106}
 \norm{\varphi_\varepsilon}_q^N=\frac{c_q}{Y_\varepsilon}+o(Y_\varepsilon^{-2}).
\end{equation}
\end{lemma}
\begin{proof}
The normalization follows from \eqref{eq:100}. Appendix~\ref{sec:A}
proves the expansions, including the cancellation of the outer energy
constant $-c_q$ with the $L^q$ contribution and the refined
$o(Y_\varepsilon^{-1})$ remainder.
\end{proof}
The first nonzero Taylor term gives the strict comparison.
\begin{proposition}\label{prop:8.2}
Assume $1<q<q_+$. Then
\begin{equation}\label{eq:107}
 \A_{N,N}>D_N\ge d(N).
\end{equation}
\end{proposition}
\begin{proof}
Since $q<q_+$, \eqref{eq:7} gives $m\le N-1$. Choose $r_0>0$ sufficiently small that $G_*>0$ on the closure of
$W_{2r_0}\setminus W_{r_0}$, which
is possible by the pole expansion. For small $\varepsilon$, the annulus is
contained in $\{G_*<t_\varepsilon\}$ and \eqref{eq:105} holds. Hence
\begin{equation}\label{eq:108}
\begin{aligned}
 \int_{W_{2r_0}\setminus W_{r_0}}\frac{\Ph(\lambda_\beta\varphi_\varepsilon^{\frac{N}{N-1}})}
 {F^{o}(x)^\beta}\dd x
 &\ge\frac{\lambda_\beta^m}{m!}Y_\varepsilon^{-m/(N-1)}
 \int_{W_{2r_0}\setminus W_{r_0}}\frac{G_*^{\frac{mN}{N-1}}}{F^{o}(x)^\beta}\dd x\\
 &=c_0Y_\varepsilon^{-m/(N-1)},
\end{aligned}
\end{equation}
where $c_0>0$. Combining \eqref{eq:104} and \eqref{eq:108},
\begin{equation}\label{eq:109}
 \J(\varphi_\varepsilon)\ge D_N+
 c_0Y_\varepsilon^{-m/(N-1)}+o(Y_\varepsilon^{-1}).
\end{equation}
For $m\le N-1$, the positive term dominates the
$o(Y_\varepsilon^{-1})$ remainder. Since $\varphi_\varepsilon\in\Sa_N$,
this proves $\A_{N,N}>D_N$; Corollary~\ref{cor:7.3} gives $D_N\ge d(N)$.
\end{proof}

\Needspace{8\baselineskip}
For $1<q<q_+$, the same test identifies the endpoint limit and
its first positive correction.

\begin{lemma}\label{lem:endpoint-test}
Choose $M\gamma>2$ in the Green test. Then the remainder in
\eqref{eq:103} is $o(Y_\varepsilon^{-2})$, and
\begin{equation}\label{eq:test-mass-precision}
 x_\varepsilon:=\norm{\varphi_\varepsilon}_q^N
 =\frac{c_q}{Y_\varepsilon}+o(Y_\varepsilon^{-3}).
\end{equation}
Put
\begin{equation}\label{eq:green-core-constant}
 C_{\rm core}:=\frac{N}{2\lambda_\beta}
 \sum_{j=1}^{N-1}\frac1{j^2}.
\end{equation}
For some $C_0>0$ independent of $\varepsilon$,
\begin{equation}\label{eq:green-core-refined}
 \int_{\{G_*>t_\varepsilon\}}
 \frac{\Ph(\lambda_\beta\varphi_\varepsilon^{N/(N-1)})}
 {F^o(z)^\beta}\dd z
 \ge D_N\left(1+\frac{C_{\rm core}}{Y_\varepsilon}
 -\frac{C_0}{Y_\varepsilon^2}\right).
\end{equation}
If $m=N-1$, the integral
\begin{equation}\label{eq:green-tail-constant}
 C_{\rm tail}:=\frac{\lambda_\beta^{N-1}}{(N-1)!D_N}
 \int_{\R^N}\frac{G_*^N}{F^o(z)^\beta}\dd z
\end{equation}
is finite and positive. After increasing $C_0$, there is $b_0>0$ such that
\begin{equation}\label{eq:green-whole-refined}
 \frac{\J(\varphi_\varepsilon)}{D_N}
 \ge 1+\frac{C_{\rm core}+C_{\rm tail}}{Y_\varepsilon}
 +b_0Y_\varepsilon^{-N/(N-1)}-C_0Y_\varepsilon^{-2}.
\end{equation}
\end{lemma}
\begin{proof}
The errors in \eqref{eq:A.11} consist of $O(R^{-\gamma})$ and
exponentially small terms. Since $R=L_\varepsilon^M$ and
$Y_\varepsilon\simeq L_\varepsilon$, the choice $M\gamma>2$ improves
the remainder to $o(Y_\varepsilon^{-2})$. The change of the $L^q$ mass
on the modified inner region is exponentially small, as shown in
Appendix~\ref{sec:A}. Dividing it by $Y_\varepsilon$ proves
\eqref{eq:test-mass-precision}.

For the inner integral, write $Y=Y_\varepsilon$ and
$U=BR^\gamma$. The calculation leading to \eqref{eq:A.12} now gives
\[
 A_\varepsilon=\frac{N-1}{\lambda_\beta}H_{N-1}+o(Y^{-2}),
 \qquad
 k_\varepsilon(s)=A_\varepsilon-
 \frac{N-1}{\lambda_\beta}\log(1+s).
\]
The probability measure
\[
 \dd\mu_N(s)=(N-1)\frac{s^{N-2}}{(1+s)^N}\dd s
\]
has finite logarithmic moments of every fixed order. Indeed, for $z<1$,
\[
 \int_0^\infty(1+s)^z\dd\mu_N(s)
 =\prod_{j=1}^{N-1}\frac{j}{j-z}.
\]
Differentiation at zero yields
\begin{equation}\label{eq:core-log-moments}
 \int\log(1+s)\dd\mu_N=H_{N-1},\qquad
 \int\bigl(\log(1+s)-H_{N-1}\bigr)^2\dd\mu_N
 =\sum_{j=1}^{N-1}\frac1{j^2}.
\end{equation}
On $0\le s\le U$, one has $|k_\varepsilon(s)|=O(\log Y)$.
Taylor's formula gives, uniformly on this interval,
\[
 \lambda_\beta\varphi_\varepsilon^{N/(N-1)}
 =\lambda_\beta Y+
 \frac{N\lambda_\beta}{N-1}k_\varepsilon(s)
 +\frac{N\lambda_\beta}{2(N-1)^2Y}k_\varepsilon(s)^2
 +O\left(\frac{|k_\varepsilon(s)|^3}{Y^2}\right).
\]
After exponentiation, the error beyond the quadratic correction is
bounded by a constant times
$(|k_\varepsilon|^3+|k_\varepsilon|^4)/Y^2$.
Use \eqref{eq:A.15} and the same change of variable as in
\eqref{eq:A.16}. The normalization error is now $o(Y^{-2})$, so
\begin{align*}
 \frac1{D_N}\int_{\{G_*>t_\varepsilon\}}
 \frac{\e^{\lambda_\beta\varphi_\varepsilon^{N/(N-1)}}}
 {F^o(z)^\beta}\dd z
 &\ge \bigl(1+o(Y^{-2})\bigr)\\
 &\quad\times
 \int_0^U\left[
 1+\frac{N\lambda_\beta k_\varepsilon(s)^2}{2(N-1)^2Y}
 +O\left(\frac{|k_\varepsilon(s)|^3+|k_\varepsilon(s)|^4}{Y^2}\right)
 \right]\dd\mu_N(s).
\end{align*}
The omitted probability tail is $O(U^{-1})=o(Y^{-2})$.
For the quadratic moment its tail, after division by $Y$, is
$O((\log U)^2/(YU))=o(Y^{-2})$. Equation
\eqref{eq:core-log-moments} therefore proves
\eqref{eq:green-core-refined} with the exponential in place of $\Ph$.
The removed Taylor polynomial has exponentially small integral on the
inner region, as in Appendix~\ref{sec:A}, and preserves this inequality.

If $m=N-1$, then $q\delta<N$. The logarithmic pole of $G_*$ and the
radial bound $G_*(r)\le Cr^{-N/q}$ at infinity show that
\eqref{eq:green-tail-constant} is finite. It is positive because $G_*$
is positive near its pole. On the outer region
$\varphi_\varepsilon=Y^{-1/N}G_*$, and the first retained Taylor term is
\[
 \frac{\lambda_\beta^{N-1}}{(N-1)!Y}
 \int_{\{G_*<t_\varepsilon\}}
 \frac{G_*^N}{F^o(z)^\beta}\dd z
 =\frac{D_NC_{\rm tail}}Y+o(Y^{-2}).
\]
The omitted moment is exponentially small by the logarithmic pole.
On a fixed Wulff annulus contained in the outer region, the next Taylor
term has integral
\[
 \frac{\lambda_\beta^N}{N!}
 Y^{-N/(N-1)}
 \int_{W_{2r_0}\setminus W_{r_0}}
 \frac{G_*^{N^2/(N-1)}}{F^o(z)^\beta}\dd z
 =D_Nb_0Y^{-N/(N-1)}
\]
with $b_0>0$. The positive Taylor series bounds the outer integral by
the sum of these two terms. Adding the inner bound
\eqref{eq:green-core-refined} proves \eqref{eq:green-whole-refined}.
\end{proof}

\begin{proposition}\label{prop:full-endpoint}
For $1<q<q_+$,
\[
 \lim_{\lambda\uparrow\lambda_N}H_N(\lambda)=d(N)=D_N,
 \qquad
 d(a)=\left(\frac aN\right)^\sigma D_N,
 \qquad
 \J(\varphi_\varepsilon)\longrightarrow D_N.
\]
Equivalently,
\[
 f(\lambda)=D_N
 \left[\frac{(\lambda/\lambda_N)^{N-1}}
 {1-(\lambda/\lambda_N)^{N-1}}\right]^\sigma(1+o(1))
 \qquad(\lambda\uparrow\lambda_N).
\]
Define
\begin{equation}\label{eq:scalar-W}
 W(x):=\frac{H_N\bigl(\lambda_N(1-x)^{1/(N-1)}\bigr)}{D_N},
 \qquad 0<x<1.
\end{equation}
Then $W$ is positive and continuous, $W(0+)=1$, $W(1-)=0$, and
\[
 \liminf_{x\downarrow0}\frac{W(x)-1}{x}
 \ge\frac{C_{\rm core}}{c_q}.
\]
If $m=N-1$, the last lower bound improves to
$(C_{\rm core}+C_{\rm tail})/c_q$.
\end{proposition}
\begin{proof}
First, every $v\in\Sa_N$ satisfies
\begin{equation}\label{eq:inverse-scalar-normalization}
 \J(v)\le H_N(\lambda_Nr^{N/(N-1)}),
 \qquad r=\norm{F(\nabla v)}_N\in(0,1).
\end{equation}
To see this directly, put $s=\norm v_q=(1-r^N)^{1/N}$,
$\rho=(r/s)^{q/N}$ and $w(z)=v(z/\rho)/r$. Then
$\norm{F(\nabla w)}_N=\norm w_q=1$, and the weighted change of variables
in \eqref{eq:11} gives \eqref{eq:inverse-scalar-normalization}.
For $v=\varphi_\varepsilon$, \eqref{eq:test-mass-precision} implies
$r\to1$. Lemma~\ref{lem:endpoint-test} and
Corollary~\ref{cor:7.3} consequently give
\[
 D_N\le\liminf_{\varepsilon\downarrow0}\J(\varphi_\varepsilon)
 \le\limsup_{\varepsilon\downarrow0}\J(\varphi_\varepsilon)
 \le d(N)\le D_N.
\]
Thus $d(N)=D_N$ and the test values converge to it.
Lemma~\ref{lem:3.1} proves the formula for $d(a)$.

The map $\varepsilon\mapsto x_\varepsilon$ is positive, continuous
and tends to zero. Indeed, $0\le h_\varepsilon(G_*)\le G_*$ gives
continuity of the $L^q$ term by dominated convergence. Coarea gives
continuity of the gradient term: the moving junction has zero measure,
and on compact intervals of positive $\varepsilon$ the inner derivatives
have a common exponential bound. The intermediate value theorem then
shows that $x_\varepsilon$ covers every sufficiently small $x>0$.
The corresponding $\varepsilon$ tends to zero with $x$, since
$x_\varepsilon$ has a positive minimum on each compact interval away
from zero.

At $x=x_\varepsilon$, \eqref{eq:inverse-scalar-normalization} gives
$D_NW(x)\ge\J(\varphi_\varepsilon)$.
Combine \eqref{eq:test-mass-precision} with
\eqref{eq:green-core-refined}, or with \eqref{eq:green-whole-refined}
when $m=N-1$, to obtain the asserted slope bounds for every sufficiently
small $x$. Since $\limsup_{x\downarrow0}W(x)=d(N)/D_N=1$, these
bounds imply $W(0+)=1$. Continuity and $W(1-)=0$ follow from
\eqref{eq:scalar-W} and Lemma~\ref{lem:2.2}.
\end{proof}

The threshold is characterized by the following scalar criterion.
\begin{proposition}\label{prop:scalar-endpoint}
Assume $1<q<q_+$. For $k\ge1$ put
\begin{equation}\label{eq:scalar-P}
 P_k(x):=\frac{1-(1-x)^k}{kx},\qquad
 \frac{H_{Nk}\bigl(\lambda_N(1-x)^{1/(N-1)}\bigr)}{d(Nk)}
 =P_k(x)^\sigma W(x).
\end{equation}
If $W(x)>1$, let $k(x)>1$ be the unique solution of
$P_{k(x)}(x)^\sigma W(x)=1$; set $k(x)=1$ when $W(x)\le1$.
Then $k(x)$ is continuous and
\begin{equation}\label{eq:scalar-finiteness}
 a_c<\infty
 \quad\Longleftrightarrow\quad
 (W(x)-1)_+=O(x)\quad(x\downarrow0),
 \qquad
 \frac{a_c}{N}=\sup_{0<x<1}k(x).
\end{equation}
If $a_c<\infty$, $\A_{a_c,N}$ is attained if and only if $k(x)$ attains
its supremum at some $x\in(0,1)$. In particular,
\begin{equation}\label{eq:scalar-strict-separation}
 \frac{a_c}{N}>\limsup_{x\downarrow0}k(x)
\end{equation}
implies attainment at $a=a_c$. If
$W(x)=1+Cx+o(x)$ for some finite $C$, then $a_c<\infty$ and
\begin{equation}\label{eq:scalar-boundary-slope}
 k(0+)=1+\frac{2C}{\sigma}.
\end{equation}
\end{proposition}
\begin{proof}
Equation \eqref{eq:scalar-P} follows from \eqref{eq:19} and
Proposition~\ref{prop:full-endpoint}. For fixed $x\in(0,1)$,
$P_1(x)=1$, $P_k(x)\to0$ as $k\to\infty$, and $P_k(x)$ strictly
decreases with $k$. Indeed, with $\ell=-\log(1-x)>0$, the sign of
$\partial_kP_k(x)$ is the sign of
$\e^{-k\ell}(1+k\ell)-1<0$. These properties prove existence,
uniqueness and continuity of $k(x)$, including across points with
$W(x)=1$.

Suppose first that $a_c<\infty$ and choose $k>a_c/N$.
Proposition~\ref{prop:3.5} and \eqref{eq:scalar-P} imply
$P_k(x)^\sigma W(x)\le1$. Since
\begin{equation}\label{eq:scalar-factor-expansion}
 P_k(x)^\sigma=1-\frac{\sigma(k-1)}2x+O_k(x^2),
\end{equation}
it follows that $(W(x)-1)_+=O(x)$.
Conversely, assume $W(x)\le1+Cx$ near zero, with $C\ge0$.
Choose $k_0>1+2C/\sigma$. Equation
\eqref{eq:scalar-factor-expansion} gives
$P_{k_0}(x)^\sigma W(x)<1$ for $0<x<x_0$, for
some $x_0>0$.
On $[x_0,1)$ the function $W$ is bounded, and
$P_k(x)\le1/(kx_0)$. Thus a sufficiently large finite $k_1\ge k_0$
makes $P_{k_1}(x)^\sigma W(x)\le1$ on the whole interval.
Its endpoint limit is one, so \eqref{eq:12} gives
$\A_{Nk_1,N}=d(Nk_1)$ and hence $a_c\le Nk_1<\infty$.
For $k\ge1$, the inequality $P_k(x)^\sigma W(x)>1$ holds at some
point precisely when $k(x)>k$ at some point. Since $a_c>N$, this proves
the supremum formula in \eqref{eq:scalar-finiteness}.

If $k(x_*)=a_c/N$ at an interior point, then
\eqref{eq:scalar-P} gives an interior maximum of $H_{a_c}$.
The lifting construction in Lemma~\ref{lem:2.2} therefore gives an
extremal, even though its value equals $d(a_c)$.
Conversely, any extremal $v$ saturates the constraint. With
$r=\norm{F(\nabla v)}_N\in(0,1)$, the same inverse normalization as
in \eqref{eq:inverse-scalar-normalization} gives
\[
 \J(v)\le H_{a_c}(\lambda_Nr^{N/(N-1)})\le\A_{a_c,N}.
\]
Both inequalities are equalities. Taking $x=1-r^N$ in
\eqref{eq:scalar-P} yields $k(x)=a_c/N$.
Since $W(1-)=0$, one has $k(x)=1$ near $x=1$.
Condition \eqref{eq:scalar-strict-separation} consequently confines a
maximizing sequence for $k$ to a compact subinterval of $(0,1)$, where
continuity gives its maximum.

Finally, if $W(x)=1+Cx+o(x)$, Proposition~\ref{prop:full-endpoint}
gives $C\ge C_{\rm core}/c_q>0$, and the finiteness criterion makes
$k(x)$ bounded. Near zero $W(x)>1$, so substitution of the expansion
and \eqref{eq:scalar-factor-expansion} into its defining equality,
uniformly for bounded $k(x)$, gives
$C-\sigma(k(x)-1)/2=o(1)$. This proves
\eqref{eq:scalar-boundary-slope}.
\end{proof}

\subsection{The first endpoint correction}\label{sec:first-correction}

Throughout this subsection $m=N-1$, and we put $p=N/(N-1)$ for the
local calculations. We analyze every maximizing branch of the
subcritical problem, rather than only a sequence realizing the upper
limsup. The full limit in Proposition~\ref{prop:full-endpoint} makes this
possible. For each $0<\lambda<\lambda_N$, choose a nonnegative
Wulff-radial maximizer $u_\lambda$ normalized by \eqref{eq:15}, and lift it
by \eqref{eq:33}--\eqref{eq:34}. With
\[
 x=1-(\lambda/\lambda_N)^{N-1},\qquad
 v_\lambda(z)=\theta_\lambda u_\lambda(\rho_\lambda z),
\]
the resulting data are
\[
 \norm{F(\nabla v_\lambda)}_N^N=1-x,
 \qquad S_\lambda=\norm{v_\lambda}_q^q=x^{q/N},
 \qquad A_\lambda=\J(v_\lambda)=D_NW(x).
\]
Let $d_\lambda=v_\lambda(0)$. Denote the nonlinear moment in
Lemma~\ref{lem:4.2} by
\[
 J_\lambda=\int_{\R^N}\lambda_\beta v_\lambda^p
 \Ph'(\lambda_\beta v_\lambda^p)F^o(z)^{-\beta}\dd z.
\]
That lemma gives a positive multiplier $\Lambda_\lambda$ and
\begin{equation}\label{eq:branch-multiplier}
 \Lambda_\lambda(1-x)=pJ_\lambda-q\delta A_\lambda.
\end{equation}
For a normalized competitor $u$, write
\[
 g_u(\lambda)=\norm{u}_q^{-q\delta}
 \int_{\R^N}\Ph(\lambda\delta |u|^p)F^o(z)^{-\beta}\dd z.
\]
The lifting gives the exact identity
\begin{equation}\label{eq:branch-derivative}
 \frac{J_\lambda}{A_\lambda}
 =\frac{\lambda g_{u_\lambda}'(\lambda)}{f(\lambda)}.
\end{equation}
Every sequence $\lambda\uparrow\lambda_N$ now realizes $d(N)$.
Lemmas~\ref{lem:4.3} and~\ref{lem:5.4} therefore imply, along every such
sequence and for every choice of radial maximizers,
\begin{equation}\label{eq:branch-leading}
 d_\lambda\longrightarrow\infty,\qquad
 \frac{J_\lambda}{\lambda_\beta d_\lambda^pA_\lambda}
 \longrightarrow1.
\end{equation}

\begin{lemma}\label{lem:branch-scales}
Uniformly over the radial maximizing branches,
\begin{align}
 (\lambda_N-\lambda)\frac{g_{u_\lambda}'(\lambda)}{f(\lambda)}
 &\longrightarrow\sigma,
 &x\frac{J_\lambda}{A_\lambda}&\longrightarrow\sigma(N-1),
 \label{eq:branch-slopes}\\
 xd_\lambda^p&\longrightarrow c_q,
 &(\lambda_N-\lambda)d_\lambda^p&\longrightarrow q/N,
 \label{eq:branch-peaks}\\
 x\frac{\Lambda_\lambda}{A_\lambda}&\longrightarrow q\delta.
 \label{eq:branch-lambda}
\end{align}
\end{lemma}
\begin{proof}
The function $f=\sup_u g_u$ is convex. At a touching maximizer,
$g_{u_\lambda}'(\lambda)$ is a supporting slope of $f$. Put
$t=\lambda_N-\lambda$ and fix $0<h<1$. Convexity gives
\[
 \frac{f(\lambda)-f(\lambda-ht)}{ht}
 \le g_{u_\lambda}'(\lambda)
 \le\frac{f(\lambda+ht)-f(\lambda)}{ht}.
\]
The exact leading asymptotic in Proposition~\ref{prop:full-endpoint}
implies
\[
 \frac{f(\lambda\pm ht)}{f(\lambda)}
 \longrightarrow(1\mp h)^{-\sigma}.
\]
After multiplication by $t/f(\lambda)$, the limiting lower and upper
bounds are
\[
 \frac{1-(1+h)^{-\sigma}}h,
 \qquad\frac{(1-h)^{-\sigma}-1}h.
\]
Both tend to $\sigma$ as $h\downarrow0$. The bounds involve $f$ alone,
so they are uniform over the active maximizers. Since
$x\lambda/(\lambda_N-\lambda)\to N-1$, \eqref{eq:branch-derivative}
proves \eqref{eq:branch-slopes}. Combine this with
\eqref{eq:branch-leading} and
$\lambda_\beta c_q=\sigma(N-1)$ to obtain the first limit in
\eqref{eq:branch-peaks}. The relation
$x/(\lambda_N-\lambda)\to(N-1)/\lambda_N$ gives the second.
Finally, \eqref{eq:branch-multiplier} gives
\eqref{eq:branch-lambda}. If one of these limits were not uniform, a
sequence of parameters and active maximizers violating it could be
selected. Applying the preceding arguments to that sequence is a
contradiction.
\end{proof}

\Needspace{9\baselineskip}
We next identify one Green limit for all branches. This only uses
qualitative compactness; a rate of convergence for the Green functions
is not assumed.

\begin{lemma}\label{lem:green-unique}
For fixed $\eta>0$, there is at most one nonnegative Wulff-radial Green
function in $L^q(\R^N)$ solving
$-Q_NG+\eta G^{q-1}=\delta_0$ with the pole, flux and decay properties
in Proposition~\ref{prop:6.3}. Moreover, for every maximizing branch,
\begin{equation}\label{eq:all-branch-green}
 G_\lambda:=d_\lambda^{1/(N-1)}v_\lambda\longrightarrow G_*
 \quad\text{in }C^1_{\mathrm{loc}}(\R^N\setminus\{0\})
 \text{ and strongly in }L^q(\R^N).
\end{equation}
Uniformly near the endpoint,
\begin{equation}\label{eq:all-branch-bounds}
 G_\lambda(r)\le C(1+|\log r|)\quad(0<r\le1),
 \qquad G_\lambda(r)\le Cr^{-N/q}\quad(r\ge1).
\end{equation}
\end{lemma}
\begin{proof}
For two Green functions $G_i(z)=g_i(r)$, $r=F^o(z)$, put
\[
 P_i(r)=N\kappa_Nr^{N-1}(-g_i'(r))^{N-1},\qquad
 h=g_1-g_2,\qquad Q=P_1-P_2.
\]
Their radial equations give
$P_i'=-\eta N\kappa_Nr^{N-1}g_i^{q-1}$. Multiplying the difference by
$h$ on $[\varepsilon,R]$ yields
\[
 \int_\varepsilon^R(-Qh')\dd r
 +\eta N\kappa_N\int_\varepsilon^R
 r^{N-1}(g_1^{q-1}-g_2^{q-1})(g_1-g_2)\dd r
 =-[Qh]_\varepsilon^R.
\]
Both integrands are nonnegative by monotonicity. Near zero, $h$ is
bounded because the logarithmic pole coefficients agree, and $Q\to0$.
At infinity, $h\to0$ and $Q\to0$. Letting
$\varepsilon\downarrow0$ and $R\uparrow\infty$, monotone convergence
shows that both integrals vanish. Strict monotonicity of the absorption
term gives $g_1=g_2$.

For the scaled maximizing branches, the absorption coefficient and norm
are
\[
 \eta_\lambda=
 \frac{q\delta A_\lambda}{S_\lambda\Lambda_\lambda}
 d_\lambda^{(N-q)/(N-1)},\qquad
 T_\lambda=\norm{G_\lambda}_q^q=(xd_\lambda^p)^{q/N}.
\]
Lemma~\ref{lem:branch-scales} and \eqref{eq:branch-leading} imply
\[
 T_\lambda\longrightarrow c_q^{q/N},\qquad
 \eta_\lambda T_\lambda
 =\frac{q\delta A_\lambda d_\lambda^p}{\Lambda_\lambda}
 \longrightarrow c_q,
 \qquad \eta_\lambda\longrightarrow c_q^{1-q/N}.
\]
Proposition~\ref{prop:6.3} and Lemma~\ref{lem:6.4} give subsequential
Green convergence and recovery of the $L^q$ norm. Every limit therefore
solves the same normalized equation \eqref{eq:91}; uniqueness identifies
it with $G_*$. This proves \eqref{eq:all-branch-green}.
The proof of Proposition~\ref{prop:6.3} bounds the radial flux on a fixed
ball by its locally bounded source mass. Together with the bounded
$L^q$ norms, this gives the logarithmic bound in
\eqref{eq:all-branch-bounds}. Radial monotonicity gives the bound at
infinity. The bounds are uniform: failure would produce a sequence of
branches to which those same compactness estimates apply.
\end{proof}

\begin{proposition}\label{prop:exact-multiplier}
The function $W$ is locally absolutely continuous on $(0,1)$. At almost
every $x$, for every active maximizer at
$\lambda=\lambda_N(1-x)^{1/(N-1)}$,
\begin{equation}\label{eq:exact-multiplier}
 \frac{\dd}{\dd x}\log W(x)
 =\frac\sigma x-\frac{\Lambda_\lambda}{NA_\lambda}.
\end{equation}
\end{proposition}
\begin{proof}
A finite convex function on an interval is locally Lipschitz. At each
differentiability point of $f$, every touching $g_{u_\lambda}$ has
derivative $f'(\lambda)$. Thus
$\lambda f'(\lambda)/f(\lambda)=J_\lambda/A_\lambda$ there. Since
\[
 W(x)=D_N^{-1}\left(\frac{x}{1-x}\right)^\sigma f(\lambda),
 \qquad\frac{\lambda'}\lambda=-\frac1{(N-1)(1-x)},
\]
we obtain
\[
 (\log W)'=
 \frac1{1-x}\left(\frac\sigma x
 -\frac{J_\lambda}{(N-1)A_\lambda}\right).
\]
Substitute \eqref{eq:branch-multiplier} and $q\delta=N\sigma$ to get
\eqref{eq:exact-multiplier}. This also establishes local absolute
continuity, without requiring a differentiable choice of maximizer.
\end{proof}

The next estimates give the first correction to the source mass. Along
an arbitrary sequence of maximizing branches approaching the endpoint,
omit the parameter subscripts on $v,d,\Lambda,A$, but retain
$S_\lambda$ for its $L^q$ mass. Set
\[
 b=\frac{p\lambda_\beta}{\Lambda},\qquad
 a_{\mathrm{abs}}=\frac{q\delta A}{S_\lambda\Lambda},\qquad
 T=\lambda_\beta d^p.
\]
Then
\begin{equation}\label{eq:source-leading-scales}
 \begin{gathered}
 -Q_Nv+a_{\mathrm{abs}}v^{q-1}
 =b v^{p-1}\Ph'(\lambda_\beta v^p)r^{-\beta},\\
 xT\longrightarrow\lambda_\beta c_q,\quad
 A\longrightarrow D_N,\quad bd^p\longrightarrow D_N^{-1},
 \quad a_{\mathrm{abs}}\simeq d^{(q-N)/(N-1)}.
 \end{gathered}
\end{equation}
where $r=F^o(z)$. Define the measure
\begin{equation}\label{eq:source-mass-exact}
 \dd\mu=bv^p\Ph'(\lambda_\beta v^p)r^{-\beta}\dd z,
 \qquad
 \mu(\R^N)=\frac{pJ_\lambda}{\Lambda}
 =1-x+\frac{q\delta A}{\Lambda}.
\end{equation}

To study its inner part, use the radial profile of $v$ and set
\[
 \rho^N=\frac{N^N\delta^N}{b\lambda_\beta^{N-2}T\e^T},
 \qquad
 \eta(s)=\lambda_\beta d^{p-1}
 \bigl(v((\rho s)^{1/\delta})-d\bigr),
 \qquad E_T(h)=T\bigl[(1+h/T)^p-1\bigr].
\]
In particular, $\rho^N\simeq\e^{-T}$. The change of radial coordinate
$t=r^\delta$ multiplies the radial $N$-Laplacian by
$\delta^Nr^{-\beta}$. Consequently,
\begin{equation}\label{eq:inner-correction-equation}
 -\Delta_N\eta
 =N^N(1+\eta/T)^{p-1}\e^{E_T(\eta)}+\mathcal R_T(s).
\end{equation}
Here $\mathcal R_T\le0$ consists of the removed Taylor polynomial and
the absorption term. On $s\le T^M$, provided $|\eta|\le C_M\log T$,
\begin{equation}\label{eq:inner-perturbation}
 |\mathcal R_T(s)|\le C_M
 \left[T^{(N-3)_+}\e^{-T}
 +T^{q-1}\e^{-T/\delta}s^{\beta/\delta}\right],
\end{equation}
with the first term omitted when $N=2$. Indeed, the removed polynomial
has degree $m-2=N-3$ when it is present. The coefficient of the absorption
term is
\[
 \frac{a_{\mathrm{abs}}\lambda_\beta^{N-1}d^q}{\delta^N}
 \rho^{N+\beta/\delta}=O(T^{q-1}\e^{-T/\delta}).
\]
Proposition~\ref{prop:5.2} and its radial integral equation give, on every
fixed interval,
\begin{equation}\label{eq:inner-fixed-limit}
 \eta(s)\longrightarrow\eta_0(s):=-(N-1)\log(1+s^p),
 \qquad -s\eta'(s)\longrightarrow\frac{Ns^p}{1+s^p}.
\end{equation}
The transformed profiles have $\eta(0)=\eta'(0)=0$.

\begin{lemma}\label{lem:growing-core}
There are fixed constants $M,\kappa>0$ with $M\kappa>2$ such that,
for $S=T^M$ and a fixed sufficiently large $R$, uniformly on $R\le s\le S$,
\begin{equation}\label{eq:growing-core-bounds}
 |\eta(s)|\le C(1+\log s),\qquad
 0<a_0\le k(s):=-s\eta'(s)\le2N,\qquad
 \e^{E_T(\eta(s))}\le Cs^{-N-\kappa}.
\end{equation}
For every fixed nonnegative integer $j$,
\begin{equation}\label{eq:growing-core-moments}
 N\int_0^S s^{N-1}\eta(s)^j\e^{E_T(\eta(s))}\dd s
 \longrightarrow
 N\int_0^\infty s^{N-1}\eta_0(s)^j\e^{p\eta_0(s)}\dd s.
\end{equation}
Moreover, $S^N\e^{E_T(\eta(S))}=o(T^{-2})$.
\end{lemma}
\begin{proof}
Choose $N-1<a_0<N$ and $1<\zeta<p$ such that
$\kappa:=\zeta a_0-N>0$, and then fix $M$ with $M\kappa>2$.
By \eqref{eq:inner-fixed-limit}, at a fixed large $R$ the initial flux
satisfies $k(R)>a_0$ with a positive margin, while
$\eta(R)=-N\log R+O(1)$. Assume $k\le2N$ on the maximal subinterval of
$[R,S]$ starting at $R$. There
\[
 |\eta(s)|\le C_R+2N\log(s/R)=O(\log T),
\]
so \eqref{eq:inner-perturbation} has an exponentially small radial
integral. The flux equation
\[
 (k^{N-1})'=s^{N-1}
 \left[N^N(1+\eta/T)^{p-1}\e^{E_T(\eta)}+\mathcal R_T\right]
\]
therefore gives $k\ge a_0$, for all sufficiently large $T$. Hence
$\eta(s)\le\eta(R)-a_0\log(s/R)$. Since $\eta\le0$ and
$\eta/T=o(1)$ uniformly on the interval, $E_T(\eta)\le\zeta\eta$.
The possible increase of $k^{N-1}$ is consequently at most
\[
 N^N\e^{\zeta\eta(R)}
 \int_R^\infty s^{N-1}(s/R)^{-\zeta a_0}\dd s
 =\frac{N^N}{\kappa}\e^{\zeta\eta(R)}R^N.
\]
This tends to zero as $R\to\infty$ because $\zeta>1$. Choosing $R$
large and then $T$ large keeps $k$ strictly below $2N$, so the bootstrap
extends to $S$. This proves \eqref{eq:growing-core-bounds}.

On $[0,S]$, fixed-radius convergence and the bootstrap give
$0\le k\le2N$. Since $N+\beta/\delta=N/\delta$,
\begin{equation}\label{eq:core-perturbation-integral}
 \int_0^S(1+k(s))|\mathcal R_T(s)|s^{N-1}\dd s
 \le C_M\left[
 T^{MN+(N-3)_+}\e^{-T}
 +T^{q-1+MN/\delta}\e^{-T/\delta}\right],
\end{equation}
again omitting the first term for $N=2$. This is smaller than every
fixed negative power of $T$. The bounds in
\eqref{eq:growing-core-bounds} give the integrable majorant
$Cs^{-1-\kappa}(1+\log s)^j$ outside $R$. Fixed-radius convergence and
dominated convergence prove \eqref{eq:growing-core-moments}.
Finally, $S^N\e^{E_T(\eta(S))}\le CS^{-\kappa}
=O(T^{-M\kappa})=o(T^{-2})$.
\end{proof}

\begin{lemma}\label{lem:inner-mass}
For $r_{\mathrm{in}}=(\rho S)^{1/\delta}$,
\begin{equation}\label{eq:inner-mass-correction}
 \mu(W_{r_{\mathrm{in}}})
 =1+\frac{N(N-1)}{2T^2}\sum_{j=1}^{N-1}\frac1{j^2}
 +o(T^{-2}).
\end{equation}
The core flux also satisfies
\begin{equation}\label{eq:inner-flux-intercept}
 k(S)=N\left(1+\frac{H_{N-1}}T+o(T^{-1})\right),
 \qquad \eta(S)+k(S)\log S=O(1).
\end{equation}
\end{lemma}
\begin{proof}
Put $z_0=k(S)/N$ within this proof, and define
\[
 J_0=N\int_0^S s^{N-1}\e^{E_T(\eta)}\dd s,
 \qquad
 J_1=N\int_0^S s^{N-1}(1+\eta/T)^{p-1}
 \e^{E_T(\eta)}\dd s.
\]
Integrating the radial flux equation and the radial Pohozaev identity
gives
\begin{equation}\label{eq:inner-pohozaev-masses}
 J_1=z_0^{N-1}+o(T^{-2}),\qquad
 J_0=z_0^N+o(T^{-2}).
\end{equation}
For the second identity, before estimating the errors the exact formula
is
\[
 J_0=z_0^N+S^N\e^{E_T(\eta(S))}
 -pN^{-N}\int_0^S k(s)s^{N-1}\mathcal R_T(s)\dd s.
\]
It follows by multiplying the flux equation by $k$ and using
$\frac{\dd}{\dd s}\e^{E_T(\eta)}
=-p(1+\eta/T)^{p-1}\e^{E_T(\eta)}k/s$.
Lemma~\ref{lem:growing-core} bounds the boundary term and
\eqref{eq:core-perturbation-integral} bounds the perturbation integral,
which proves \eqref{eq:inner-pohozaev-masses} with the stated accuracy.

The probability measure $Ns^{N-1}(1+s^p)^{-N}\dd s$ becomes
$(N-1)u^{N-2}(1+u)^{-N}\dd u$ under $u=s^p$. Its first two
$\eta_0$ moments, denoted by $M_1,M_2$ in this calculation, are
\[
 M_1=-(N-1)H_{N-1},\qquad
 M_2=(N-1)^2\left(H_{N-1}^2+
 \sum_{j=1}^{N-1}\frac1{j^2}\right).
\]
These follow by setting $z=u/(1+u)$ and differentiating
\[
 (N-1)\int_0^1 z^{N-2}(1-z)^t\dd z
 =\prod_{j=1}^{N-1}\frac j{j+t}
\]
at $t=0$.
Lemma~\ref{lem:growing-core} first gives $J_1\to1$, hence $z_0\to1$.
Expanding $J_1-J_0$ once, using the bounded second moment, yields
\[
 z_0^{N-1}-z_0^N
 =\frac{p-1}{T}\bigl(M_1+o(1)\bigr)+O(T^{-2}).
\]
Thus $z_0=1+H_{N-1}/T+o(T^{-1})$.
The flux equation and \eqref{eq:growing-core-bounds} also give
\[
 |k(S)-k(s)|\le Cs^{-\kappa}+o(T^{-3})\qquad(R\le s\le S).
\]
Integrating against $\dd s/s$ proves the bounded intercept in
\eqref{eq:inner-flux-intercept}.

The change of radial variable in \eqref{eq:source-mass-exact} has
coefficient
$N^N(N\kappa_N)\delta^{N-1}/\lambda_\beta^{N-1}=N$.
The removed Taylor polynomial contributes only an exponentially small
error, so
\[
 \mu(W_{r_{\mathrm{in}}})
 =N\int_0^S s^{N-1}(1+\eta/T)^p\e^{E_T(\eta)}\dd s
 +o(T^{-2}).
\]
For small $h$,
\[
 (1+h)^p-N(1+h)^{p-1}+(N-1)
 =\frac p2h^2+O(|h|^3).
\]
There is no linear term. The bounded third moment,
\eqref{eq:inner-pohozaev-masses} and
\eqref{eq:growing-core-moments} give
\begin{align*}
 \mu(W_{r_{\mathrm{in}}})
 &=Nz_0^{N-1}-(N-1)z_0^N+\frac{pM_2}{2T^2}+o(T^{-2})\\
 &=1-\frac{N(N-1)H_{N-1}^2}{2T^2}
 +\frac{pM_2}{2T^2}+o(T^{-2}),
\end{align*}
which is \eqref{eq:inner-mass-correction}. The calculation uses only
the limiting moments and the two radial identities, and does not
assume a higher-order expansion of the inner profile.
\end{proof}

We now choose a second, outer cutoff and estimate the entire intervening
neck. Write $G=d^{1/(N-1)}v$ and put $r_{\mathrm{out}}=d^{-L}$, with
$L$ a sufficiently large fixed constant. By
\eqref{eq:all-branch-bounds}, $v=o(1)$ uniformly on
$r\ge r_{\mathrm{out}}$. Moreover, since $N>q\delta$, dominated
convergence gives
\begin{equation}\label{eq:outer-green-moment}
 \int_{\R^N}G^Nr^{-\beta}\dd z
 \longrightarrow\int_{\R^N}G_*^Nr^{-\beta}\dd z.
\end{equation}
Indeed, the logarithmic majorant is integrable at zero, and the radial
power bound is integrable at infinity against this weight.

Fix a small $r_0>0$. The source measure in the equation for $G$ has
mass $O(d^{-N/(N-1)})$ on
$W_{r_0}\setminus W_{r_{\mathrm{out}}}$: its density there is bounded by
$Cd^{-N/(N-1)}G^{N-1}r^{-\beta}$. If
$P(r)=N\kappa_Nr^{N-1}(-G'(r))^{N-1}$, its radial equation and
convergence at $r_0$ imply
\[
 P(r)\ge P(r_0)-o(1)\ge c_0>0
 \qquad(r_{\mathrm{out}}\le r\le r_0).
\]
Consequently,
\begin{equation}\label{eq:outer-cutoff-values}
 G(r_{\mathrm{out}})\ge c_1L\log d-O(1),
 \qquad G(r_{\mathrm{out}})\le CL(1+\log d),
\end{equation}
where $c_1$ is independent of the chosen fixed $L$.

On $r\le r_{\mathrm{out}}$, the ratio of the source to the absorption
term in \eqref{eq:source-leading-scales} is at least
\[
 Cd^{-q/(N-1)}r^{-\beta}v^{N-q}.
\]
Using $c\log d/d^{1/(N-1)}\le v\le d$ shows that this ratio tends
uniformly to infinity if
\begin{equation}\label{eq:outer-cutoff-choice}
 L\beta>\max\left\{\frac{N}{N-1},\ p(q-N+1)\right\}.
\end{equation}
The first exponent applies when $N-q\ge0$ and the second when
$N-q<0$. Thus $-Q_Nv\ge0$ on $r\le r_{\mathrm{out}}$ for all
sufficiently large $T$. This is the step that uses $\beta>0$.

\begin{lemma}\label{lem:neck-mass}
By increasing the fixed $L$ in \eqref{eq:outer-cutoff-choice} if
necessary,
\begin{equation}\label{eq:neck-mass}
 \mu(W_{r_{\mathrm{out}}}\setminus W_{r_{\mathrm{in}}})=o(T^{-2}).
\end{equation}
\end{lemma}
\begin{proof}
In inner coordinates let $s_{\mathrm{out}}=r_{\mathrm{out}}^\delta/\rho$.
Superharmonicity makes $k(s)=-s\eta'(s)$ nondecreasing on
$[S,s_{\mathrm{out}}]$. Put $k_S=k(S)$ and
$b_T=\eta(S)+k_S\log S=O(1)$. Then
\[
 0<1+\eta(s)/T\le z(s):=
 1+b_T/T-k_S\log s/T.
\]
By \eqref{eq:outer-cutoff-values},
$z(s_{\mathrm{out}})\ge A_L\log T/T$, where $A_L$ can be made
arbitrarily large by choosing $L$ large. Since $\Ph'\le\exp$,
\begin{equation}\label{eq:neck-integral}
 \begin{gathered}
 \mu(W_{r_{\mathrm{out}}}\setminus W_{r_{\mathrm{in}}})
 \le\frac{NT}{k_S}
 \int_{z(s_{\mathrm{out}})}^{z(S)}z^p\e^{E(z)}\dd z,\\
 E(z)=T\left[\frac N{k_S}(1+b_T/T-z)+z^p-1\right].
 \end{gathered}
\end{equation}
Lemma~\ref{lem:inner-mass} gives $E(0)=O(1)$. On a fixed small
interval $0\le z\le z_0$, one has $E(z)\le C-Tz/4$. Its contribution
to \eqref{eq:neck-integral} is bounded by
\[
 CT^{-p}\int_{A_L\log T}^\infty w^p\e^{-w/4}\dd w=o(T^{-2}),
\]
provided $A_L/4>2-p$. On a fixed middle subinterval of $(0,1)$,
$z^p-z<0$, so the contribution decays exponentially in $T$.

Choose $z_1<1$ close enough to one. On the remaining interval
$z_1\le z\le z(S)$, $E'(z)\ge cT$, whereas
\[
 E(z(S))=N\log S+E_T(\eta(S))=-p\log S+O(1).
\]
Here \eqref{eq:inner-flux-intercept} gives
$\eta(S)=-N\log S+O(1)$, and the Taylor remainder in
$E_T(\eta(S))$ is $O((\log T)^2/T)$. This part of the integral is
$O(S^{-p})=o(T^{-2})$, since $M\kappa>2$ and $\kappa<p$.
If one of the three intervals is empty, its contribution is zero.
Together these bounds prove \eqref{eq:neck-mass}.
\end{proof}

\begin{theorem}\label{thm:first-correction}
If $m=N-1$, then
\begin{equation}\label{eq:first-endpoint-correction}
 W(x)=1+\frac{C_{\mathrm{core}}+C_{\mathrm{tail}}}{c_q}x+o(x)
 \qquad(x\downarrow0).
\end{equation}
In particular, the threshold $a_c$ is finite. The source estimates used
to obtain this expansion are uniform over all radial maximizing branches.
\end{theorem}
\begin{proof}
On $r\ge r_{\mathrm{out}}$, the function $v$ is uniformly small.
Since $m=N-1$,
\[
 \Ph'(s)=\frac{s^{N-2}}{(N-2)!}(1+O(s))\qquad(s\downarrow0).
\]
Thus \eqref{eq:outer-green-moment}, $bd^p\to D_N^{-1}$ and
$T=\lambda_\beta d^p$ give
\begin{equation}\label{eq:outer-mass-correction}
 \mu(\R^N\setminus W_{r_{\mathrm{out}}})
 =\frac{\lambda_\beta^N}{(N-2)!D_NT^2}
 \int_{\R^N}G_*^Nr^{-\beta}\dd z+o(T^{-2})
 =\frac{\lambda_\beta(N-1)C_{\mathrm{tail}}}{T^2}
 +o(T^{-2}).
\end{equation}
The omitted Green moment inside $W_{r_{\mathrm{out}}}$ tends to zero
by the same logarithmic majorant. The three regions are disjoint.
Lemmas~\ref{lem:inner-mass} and~\ref{lem:neck-mass}, together with
\eqref{eq:outer-mass-correction}, therefore imply
\begin{equation}\label{eq:total-mass-correction}
 e_T:=\mu(\R^N)-1
 =\frac{\lambda_\beta(N-1)
 (C_{\mathrm{core}}+C_{\mathrm{tail}})}{T^2}+o(T^{-2}).
\end{equation}
Equation~\eqref{eq:source-mass-exact} gives the exact relation
$\Lambda/A=q\delta/(x+e_T)$. At almost every parameter,
\eqref{eq:exact-multiplier} hence becomes
\[
 (\log W)'(x)=\frac{\sigma e_T}{x(x+e_T)}.
\]
Since $xT\to\lambda_\beta c_q$ and $e_T=o(x)$, its right-hand side
tends to
\[
 \frac{\sigma\lambda_\beta(N-1)
 (C_{\mathrm{core}}+C_{\mathrm{tail}})}{(\lambda_\beta c_q)^2}
 =\frac{C_{\mathrm{core}}+C_{\mathrm{tail}}}{c_q}.
\]
All estimates hold along every sequence of active radial maximizers,
so their remainders are uniform over the branches. In particular the
derivative is bounded near zero. Local absolute continuity and
$W(0+)=1$ allow integration down to zero, which proves
\eqref{eq:first-endpoint-correction}.
The finiteness criterion in Proposition~\ref{prop:scalar-endpoint}
now applies. No quantitative rate in \eqref{eq:all-branch-green} was
used: the first correction follows from the radial identities, the
growing-core bounds and the full neck estimate.
\end{proof}

\subsection{Attainment at the threshold}\label{sec:threshold-attainment}
In dimensions $N\ge3$, the next outer Taylor term is larger than the
quadratic normalization error. This gives a strict separation from the
concentrating endpoint and settles attainment at the threshold.
\begin{proposition}\label{prop:highdim-endpoint}
Assume $N\ge3$, $q>1$, and $q_-\le q<q_+$. Put
\begin{equation}\label{eq:threshold-test-exponent}
 a_*:=N\left(1+
 \frac{2(C_{\rm core}+C_{\rm tail})}{\sigma c_q}\right).
\end{equation}
Then
\[
 \A_{a_*,N}>d(a_*),\qquad a_*<a_c<\infty,
 \qquad \A_{a_c,N}\text{ is attained}.
\]
\end{proposition}
\begin{proof}
Here $m=N-1$. Let $k_*=a_*/N$ and use the test from
Lemma~\ref{lem:endpoint-test}. The spatial dilation
$v_\varepsilon=T_{N,a_*}\varphi_\varepsilon$ belongs to $\Sa_{a_*}$.
Lemma~\ref{lem:3.3} and Proposition~\ref{prop:full-endpoint} give
\[
 \frac{\J(v_\varepsilon)}{d(a_*)}
 =P_{k_*}(x_\varepsilon)^\sigma
 \frac{\J(\varphi_\varepsilon)}{D_N}.
\]
By \eqref{eq:test-mass-precision} and
\eqref{eq:scalar-factor-expansion},
\[
 P_{k_*}(x_\varepsilon)^\sigma
 =1-\frac{\sigma(k_*-1)c_q}{2Y_\varepsilon}
 +O(Y_\varepsilon^{-2})
 =1-\frac{C_{\rm core}+C_{\rm tail}}{Y_\varepsilon}
 +O(Y_\varepsilon^{-2}).
\]
The first-order terms cancel in \eqref{eq:green-whole-refined}, leaving
\[
 \frac{\J(v_\varepsilon)}{d(a_*)}
 \ge1+b_0Y_\varepsilon^{-N/(N-1)}-CY_\varepsilon^{-2}.
\]
Since $N/(N-1)<2$, this is strictly larger than one for all sufficiently
small $\varepsilon$. Thus $\A_{a_*,N}>d(a_*)$, and continuity in
Proposition~\ref{prop:3.4} implies $a_c>a_*$.

Theorem~\ref{thm:first-correction} gives
\[
 W(x)=1+\frac{C_{\rm core}+C_{\rm tail}}{c_q}x+o(x).
\]
Proposition~\ref{prop:scalar-endpoint} therefore implies $a_c<\infty$
and
\[
 \lim_{x\downarrow0}k(x)
 =1+\frac{2(C_{\rm core}+C_{\rm tail})}{\sigma c_q}
 =\frac{a_*}{N}<\frac{a_c}{N}.
\]
The strict separation criterion
\eqref{eq:scalar-strict-separation} yields attainment at $a_c$.
\end{proof}

It remains to settle dimension two. Throughout the rest of this subsection,
$N=2$, $0<\beta<2$ and $1<q<q_+=4/(2-\beta)$, so that $m=1$ and
$0<\sigma=q\delta/2<1$. Set
\[
 R_2:=\frac{\lambda_\beta^2}{D_2}
       \int_{\R^2}G_*^2F^o(z)^{-\beta}\dd z,
 \qquad
 R_4:=\frac{\lambda_\beta^4}{2D_2}
       \int_{\R^2}G_*^4F^o(z)^{-\beta}\dd z.
\]
Both moments are finite: the pole is logarithmic and the radial decay
bound controls infinity, since $q\delta<2$. By
Theorem~\ref{thm:first-correction},
\begin{equation}\label{eq:planar-slope}
 W(x)=1+\frac{1+R_2}{\sigma}x+o(x),
 \qquad
 \lim_{x\downarrow0}k(x)=\frac{a_*}{2},
 \qquad
 a_*:=2\left(1+\frac{2(1+R_2)}{\sigma^2}\right).
\end{equation}
Here $k(x)$ is the scalar threshold function from
Proposition~\ref{prop:scalar-endpoint}. In particular, $a_c<\infty$.
We shall obtain $\A_{a_*,2}>d(a_*)$. Together with
\eqref{eq:planar-slope}, this prevents a maximizing sequence for $k$ from
escaping to $x=0$ and proves attainment at $a_c$.

\begin{lemma}\label{lem:planar-core}
Choose the cutoff exponent in \eqref{eq:profile} so that $M\gamma>2$.
Then the Green test satisfies
\begin{equation}\label{eq:planar-core}
 \frac1{D_2}\int_{\{G_*>t_\varepsilon\}}
 \Ph(\lambda_\beta\varphi_\varepsilon^2)F^o(z)^{-\beta}\dd z
 \ge 1+\frac{1}{\lambda_\beta Y_\varepsilon}
       +\frac{9}{2\lambda_\beta^2Y_\varepsilon^2}
       +o(Y_\varepsilon^{-2}).
\end{equation}
\end{lemma}
\begin{proof}
The normalization proof in Appendix~\ref{sec:A}, with $M\gamma>2$,
improves its error to $o(Y_\varepsilon^{-2})$ and gives
$\norm{\varphi_\varepsilon}_q^2=c_q/Y_\varepsilon+o(Y_\varepsilon^{-3})$.
For the probability measure $\dd\mu_2(s)=(1+s)^{-2}\dd s$, the variable
$t=\log(1+s)$ has density $\e^{-t}\dd t$. Thus, with
$k(s)=(1-\log(1+s))/\lambda_\beta$,
\[
 \int k\dd\mu_2=0,
 \qquad \int k^2\dd\mu_2=\lambda_\beta^{-2},
 \qquad \int k^4\dd\mu_2=9\lambda_\beta^{-4}.
\]
In the core change of variables from Appendix~\ref{sec:A}, the improved
normalization gives $k_\varepsilon=k+o(Y_\varepsilon^{-2})$ uniformly.
Since the exponent is quadratic, the remaining factor is
$\e^{\lambda_\beta k(s)^2/Y_\varepsilon}$, with scalar error
$o(Y_\varepsilon^{-2})$. Its first two integrated coefficients are
$\lambda_\beta\int k^2\dd\mu_2$ and
$\lambda_\beta^2\int k^4\dd\mu_2/2$.
The uniform Taylor remainder is
$O((\log Y_\varepsilon)^6/Y_\varepsilon^3)=o(Y_\varepsilon^{-2})$.
The omitted probability is $O(L_\varepsilon^{-M\gamma})$
and the omitted second moment is
$O(L_\varepsilon^{-M\gamma}(\log L_\varepsilon)^2)$;
the latter is multiplied by $Y_\varepsilon^{-1}$.
The fourth moment converges. These errors have the required size.
The subtracted constant has exponentially small weighted integral on
the core, which proves \eqref{eq:planar-core}.
\end{proof}

\begin{proposition}\label{prop:planar-optimized}
Define
\[
 \widehat Q:=\frac{9+\sigma}{2}
 +\frac{3\sigma^2-9\sigma+4}{6\sigma}(1+R_2)^2
 +\left(1-\frac{4\sigma}{3}\right)(1+R_2)+R_4.
\]
If $\widehat Q>0$, then $a_*<a_c<\infty$ and
$\A_{a_c,2}$ is attained.
\end{proposition}
\begin{proof}
Write $Y_0=Y_\varepsilon$ and
\[
 C=\frac{1+R_2}{\lambda_\beta},\qquad
 C_{\rm tail}=\frac{R_2}{\lambda_\beta},\qquad
 C_2=\frac{9}{2\lambda_\beta^2}.
\]
For $S$ near $c_q$, let
$G^{(S)}(z)=G_*(\tau_Sz)$ with
$\tau_S=(c_q/S)^{q/4}$. Then
$\norm{G^{(S)}}_q^2=S$ and the pole constant satisfies
$\lambda_\beta(A(S)-A_*)=\sigma\log(S/c_q)$.
Normalize $h_\varepsilon(G^{(S)})$ by
\[
 Y=\norm{F(\nabla h_\varepsilon(G^{(S)}))}_2^2
      +\norm{h_\varepsilon(G^{(S)})}_q^2,
 \qquad
 \varphi_{\varepsilon,S}=Y^{-1/2}h_\varepsilon(G^{(S)}).
\]
Spatial dilation preserves the energy and multiplies the squared lower
norm by $S/c_q$. Hence, uniformly for $S$ in a fixed compact
neighborhood of $c_q$,
\[
 Y=Y_0+S-c_q+o(Y_0^{-2}),\qquad
 x_{\varepsilon,S}:=\norm{\varphi_{\varepsilon,S}}_q^2
                  =\frac{S}{Y}+o(Y^{-3}).
\]
Take $S=c_q+t/Y_0$, with $t\in\R$ fixed. In the core calculation,
$k(s)$ is replaced by $k(s)-(S-c_q)$, up to the same negligible
error. Since its mean is zero, its second moment changes by
$O((S-c_q)^2)$, which contributes only $O(Y^{-3})$.
Its fourth moment has the same limit. The core prefactor is
\[
 L(S)=\left(\frac{S}{c_q}\right)^\sigma
       \exp\left[-\frac{\sigma}{c_q}(S-c_q)\right],
\]
where $\lambda_\beta c_q=\sigma$. Consequently the core contribution,
divided by $D_2$, is bounded below by
$L(S)(1+1/(\lambda_\beta Y)+C_2/Y^2+o(Y^{-2}))$.
These errors are uniform near $c_q$: dilation only introduces bounded
factors and preserves the cutoff in the Green variable.

On the outer region the first Taylor term contributes
$C_{\rm tail}(S/c_q)^\sigma/Y+o(Y^{-2})$, and the second contributes
\[
 \frac{\lambda_\beta^2}{2D_2Y^2}
 \int_{\R^2}(G^{(S)})^4F^o(z)^{-\beta}\dd z
 =\frac{R_4}{\lambda_\beta^2Y^2}+o(Y^{-2}).
\]
The omitted pole moments are exponentially small. Using
$S-c_q=t/Y+O(Y^{-3})$, we obtain
\begin{equation}\label{eq:planar-dilated-test}
 \frac{\J(\varphi_{\varepsilon,S})}{D_2}
 \ge 1+\frac{C}{Y}
 +\frac{C_2+R_4/\lambda_\beta^2
        +\sigma C_{\rm tail}t/c_q-\sigma t^2/(2c_q^2)}{Y^2}
 +o(Y^{-2}).
\end{equation}
Apply the exact spatial dilation from $a=2$ to $a=a_*$. For
$k=a_*/2$, its normalized multiplier expands as
\[
 P_k(x)^\sigma=1-\frac{\sigma(k-1)}2x
 +\left[\frac{\sigma(k-1)(k-2)}6
       +\frac{\sigma(\sigma-1)(k-1)^2}8\right]x^2+O_k(x^3).
\]
At $x=x_{\varepsilon,S}$ the linear term cancels that in
\eqref{eq:planar-dilated-test}. The coefficient of $Y^{-2}$ becomes
\[
 \frac{Q+R_4}{\lambda_\beta^2}
 +\frac{\sigma C_{\rm tail}-C}{c_q}t
 -\frac{\sigma}{2c_q^2}t^2,
 \qquad
 Q=\frac92+\frac{1-3\sigma}{6\sigma}(1+R_2)^2
              -\frac{\sigma(1+R_2)}3.
\]
Its maximum, attained at
$t=c_q(\sigma C_{\rm tail}-C)/\sigma$, equals
$\widehat Q/\lambda_\beta^2$, because
\[
 \widehat Q=Q+\frac{((\sigma-1)R_2-1)^2}{2\sigma}+R_4.
\]
Positivity therefore gives $\A_{a_*,2}>d(a_*)$ and $a_c>a_*$.
The endpoint limit in \eqref{eq:planar-slope} and
Proposition~\ref{prop:scalar-endpoint} give attainment at $a_c$.
\end{proof}

For part of the parameter range, independent variations on distant
annuli give a stronger comparison. By Proposition~\ref{prop:2.1}, we
may work with $F(\xi)=|\xi|$. Assume for the moment that
$q>2$ and
write $G=G_*$, $I=\int_{\R^2}G^q\dd z$ and
$\eta=I^{2/q-1}$. For a real radial
$\psi\in C_c^\infty(\R^2\setminus\{0\})$, put
\[
 A_\psi=\int G^{q-1}\psi\dd z,\qquad
 B_\psi=\int G^{q-2}\psi^2\dd z,\qquad
 s_\psi=2\eta A_\psi,
\]
\[
 \begin{aligned}
 \mathcal E(\psi)&=\int|\nabla\psi|^2\dd z
 +(q-1)\eta B_\psi+(2-q)\eta A_\psi^2/I,\\
 \mathcal L(\psi)&=\frac{2\lambda_\beta}{D_2}
       \int G\psi|z|^{-\beta}\dd z-\frac{C}{c_q}s_\psi,
 \qquad C=\frac{1+R_2}{\lambda_\beta}.
 \end{aligned}
\]
All unweighted integrals in these formulas are over $\R^2$.
For nonzero $\psi$, $\mathcal E(\psi)>0$: Cauchy--Schwarz gives
$A_\psi^2\le IB_\psi$, so the last two terms are at least
$\eta B_\psi$.

\begin{lemma}\label{lem:planar-variation}
Assume $q>2$. For every such $\psi$ and every fixed $t\in\R$,
\begin{equation}\label{eq:planar-variation}
 \liminf_{x\downarrow0}
 \frac{W(x)-1-(C/c_q)x}{x^2}
 \ge\frac{C_2+t\mathcal L(\psi)
                 -\lambda_\beta t^2\mathcal E(\psi)}{c_q^2},
 \qquad C_2=\frac{9}{2\lambda_\beta^2}.
\end{equation}
In particular the numerator can be increased from $C_2$ to
$C_2+\mathcal L(\psi)^2/(4\lambda_\beta\mathcal E(\psi))$.
\end{lemma}
\begin{proof}
The radial Green profile is positive at every finite radius. Indeed, a
zero would also have zero derivative by nonnegativity; initial-value
uniqueness for $g''+g'/r=\eta|g|^{q-2}g$ would then contradict its
logarithmic pole. Thus $G$ has a positive minimum on the support of
$\psi$.
Let $w_\varepsilon=h_\varepsilon(G)$ and let $Y_0$ be its normalization
in \eqref{eq:100}, with $M\gamma>2$. For small $\varepsilon$,
$w_\varepsilon=G$ on that fixed support. Set
\[
 w_{\varepsilon,t}=w_\varepsilon+\frac{t}{Y_0}\psi,
 \quad
 Y=\int|\nabla w_{\varepsilon,t}|^2\dd z
       +\norm{w_{\varepsilon,t}}_q^2,
 \quad
 v_{\varepsilon,t}=Y^{-1/2}w_{\varepsilon,t}.
\]
These functions are nonnegative for small $\varepsilon$.
The Green equation gives $\int\nabla G\cdot\nabla\psi\dd z=-\eta A_\psi$.
Taylor expansion on the fixed support gives, as $\zeta\to0$,
\[
 \norm{G+\zeta\psi}_q^2
 =c_q+2\eta A_\psi\zeta
 +\bigl[(q-1)\eta B_\psi+(2-q)\eta A_\psi^2/I\bigr]\zeta^2
 +o(\zeta^2).
\]
The same expansion holds with $G$ replaced by $w_\varepsilon$;
its unperturbed integral differs from $I$ by an exponentially small
error. The linear energy and norm changes therefore cancel, and
\[
 Y=Y_0+\frac{t^2\mathcal E(\psi)}{Y_0^2}+o(Y_0^{-2}),
 \qquad
 x_{\varepsilon,t}:=\norm{v_{\varepsilon,t}}_q^2
 =\frac{c_q}{Y}+\frac{ts_\psi}{Y^2}+o(Y^{-2}).
\]
The unnormalized core is unchanged, where
$w_\varepsilon=Y_0+O(\log Y_0)$ uniformly. The change of normalization
multiplies its exponential by
\[
 \exp\bigl[\lambda_\beta w_\varepsilon^2(Y^{-1}-Y_0^{-1})\bigr]
 =1-\frac{\lambda_\beta t^2\mathcal E(\psi)}{Y^2}+o(Y^{-2}).
\]
Lemma~\ref{lem:planar-core} consequently gives a core contribution at
least
$1+1/(\lambda_\beta Y)+(C_2-\lambda_\beta t^2\mathcal E(\psi))/Y^2
+o(Y^{-2})$, after division by $D_2$.
On the outer region retain only the first Taylor term. Its contribution
is
\[
 \frac{C_{\rm tail}}{Y}
 +\frac{2\lambda_\beta t}{D_2Y^2}
     \int G\psi|z|^{-\beta}\dd z+o(Y^{-2}).
\]
The removed pole moments are exponentially small. Adding the two
disjoint estimates and subtracting $(C/c_q)x_{\varepsilon,t}$ gives
\eqref{eq:planar-variation} along these competitors.
To obtain the bound at every small $x$, note that for fixed $t,\psi$
the map $\varepsilon\mapsto x_{\varepsilon,t}$ is continuous, positive
and tends to zero. Its image contains an interval $(0,x_0)$ by the
intermediate value theorem. The chosen $\varepsilon$ tends to zero
with $x$, since this map has a positive minimum on each compact
interval of positive $\varepsilon$. The competitor inequality
$\J(v_{\varepsilon,t})\le D_2W(x_{\varepsilon,t})$ does not require
radial monotonicity. Since $x_{\varepsilon,t}Y\to c_q$, this proves
\eqref{eq:planar-variation}. Optimizing its concave quadratic in $t$
gives the last assertion.
\end{proof}

\begin{lemma}\label{lem:planar-tail}
If $q>2$ and $\alpha=2/(q-2)$, there are positive constants
$c_-,c_+$ such that
\[
 c_-r^{-\alpha}\le g_*(r)\le c_+r^{-\alpha}\qquad(r\ge1).
\]
\end{lemma}
\begin{proof}
For $w(r)=cr^{-\alpha}$,
\[
 -\Delta w+\eta w^{q-1}
 =c(\eta c^{q-2}-\alpha^2)r^{-\alpha-2}.
\]
Choose $c_-$ no larger than $g_*(1)$ and
$(\alpha^2/\eta)^{1/(q-2)}$, and choose $c_+$ no smaller than both.
The corresponding powers are a subsolution and a supersolution outside
the unit disk. They and $g_*$ tend to zero at infinity. The comparison
principle, or a positive interior maximum of their difference and the
strict increase of $s^{q-1}$, gives the conclusion.
\end{proof}

\begin{proposition}\label{prop:planar-distant}
If $\delta(q-2)\ge1$, then
\begin{equation}\label{eq:planar-infinite-remainder}
 \lim_{x\downarrow0}
 \frac{W(x)-1-(1+R_2)x/\sigma}{x^2}=+\infty.
\end{equation}
Consequently $a_*<a_c<\infty$ and $\A_{a_c,2}$ is attained.
\end{proposition}
\begin{proof}
The assumptions $q\delta<2$ and $\delta(q-2)\ge1$ imply $q>4$.
Fix a nonzero, nonnegative smooth radial function $\chi$ supported in
$1<|z|<2$, and set $\psi_L(z)=\chi(z/L)$ for $L\ge1$.
Lemma~\ref{lem:planar-tail} and a change of variables give constants
independent of $L$ such that
\[
 \int G\psi_L|z|^{-\beta}\dd z\ge cL^{2\delta-\alpha},
 \quad 0\le A_{\psi_L}\le CL^{-\alpha},\quad
 B_{\psi_L}\le C,\quad
 \int|\nabla\psi_L|^2\dd z=\int|\nabla\chi|^2\dd z.
\]
Here $2-\alpha(q-1)=-\alpha$ and $\alpha(q-2)=2$.
The rank-one term in $\mathcal E$ is nonpositive, so
\[
 0<\mathcal E(\psi_L)\le C,
 \qquad
 \mathcal L(\psi_L)\ge cL^{2\delta-\alpha}-CL^{-\alpha}.
\]
If $\delta(q-2)>1$, the quotient
$\mathcal L(\psi_L)^2/\mathcal E(\psi_L)$ is unbounded.
If $\delta(q-2)=1$, take disjoint annuli $L_j=3^jL_0$ and
$\psi^{(J)}=\sum_{j=1}^J\psi_{L_j}$. Their weighted linear
contributions sum to at least $cJ$, whereas
$A_{\psi^{(J)}}\le C\sum_jL_j^{-\alpha}\le C$.
The energy and $B$ terms add on the disjoint supports, and the
rank-one term remains nonpositive. Thus
\[
 \mathcal L(\psi^{(J)})\ge cJ-C,
 \qquad 0<\mathcal E(\psi^{(J)})\le CJ,
\]
and the same quotient is unbounded.
Given any finite target bound, first fix a compactly supported member
of this family and its optimizing scalar $t$ in
Lemma~\ref{lem:planar-variation}, and then let $\varepsilon\downarrow0$.
The support stays fixed during each Taylor expansion. Increasing the
target bound afterwards proves \eqref{eq:planar-infinite-remainder}
without exchanging these limits.
At $k=a_*/2$, the linear term of $P_k(x)^\sigma W(x)-1$ cancels,
and its quadratic quotient tends to $+\infty$ by
\eqref{eq:planar-infinite-remainder}. Hence
$\A_{a_*,2}>d(a_*)$. Equation~\eqref{eq:planar-slope} and
Proposition~\ref{prop:scalar-endpoint} give the conclusion.
\end{proof}

We treat the remaining parameters by bounding the Green moments.
Continue in the Euclidean normalization, now allowing every $q>1$,
and put
\[
 b_q=\begin{cases}4/(q+1),&1<q\le5,\\
                   2/(q-2),&q>5.
      \end{cases}
\]
\begin{lemma}\label{lem:planar-barrier}
On the positive branch of the radial Green profile, use
$w=2\pi g_*(r)$ as the level variable and write
$P(w)=-2\pi r g_*'(r)$. Then
\[
 P(w)\ge1-\e^{-b_qw}\qquad(w>0).
\]
For $V(w)=\int_{\{2\pi G_*>w\}}|z|^{-\beta}\dd z$, it follows that
\begin{equation}\label{eq:planar-volume}
 \frac{D_2}{\e}\e^{-2\delta w}
 \le V(w)\le
 \frac{D_2}{\e}(\e^{b_qw}-1)^{-2\delta/b_q}.
\end{equation}
\end{lemma}
\begin{proof}
If $G_*$ has finite support, its inverse radius is used only at positive
levels; no positivity outside that support is needed. The Green equation
and flux limits give $P(0+)=0$, $P(\infty)=1$ and $0<P<1$.
With $\eta=\norm{G_*}_q^{2-q}$, differentiation gives
\[
 PP'=\eta(2\pi)^{2-q}r(w)^2w^{q-1}>0,
 \qquad \frac{r'}r=-\frac1P,
\]
and hence
\begin{equation}\label{eq:planar-flux-ode}
 PP''+(P')^2+2P'-\frac{q-1}{w}PP'=0.
\end{equation}
For $b=b_q$ and $L(w)=1-\e^{-bw}$, the left-hand side equals
\[
 b\e^{-bw}\bigl[2-bH_q(bw)\bigr],
 \qquad
 H_q(z)=1-2\e^{-z}+(q-1)\frac{1-\e^{-z}}z.
\]
One has $H_1\le1$ and $H_5\le3$, the latter following from
$2\tanh(z/2)\le z$. Since $H_q$ is affine in $q$, these estimates
give $H_q\le(q+1)/2$ for $1<q\le5$. For $q>5$,
$H_q=H_5+(q-5)(1-\e^{-z})/z\le q-2$.
Thus the differential expression at $L$ is nonnegative.
If $P-L$ were negative, it would attain a negative minimum at a finite
positive $w$, since its limits at both ends are zero. There
$P'=L'$ and $P''\ge L''$, so the left-hand side of
\eqref{eq:planar-flux-ode} would be at least
\[
 \left[LL''+(L')^2+2L'-\frac{q-1}{w}LL'\right]
 +(P-L)\left[L''-\frac{q-1}{w}L'\right]>0,
\]
a contradiction. This proves the flux bound.
Set $T(w)=-\log r(w)+2\pi A_*$. The pole gives
$T(w)-w\to0$ as $w\to\infty$, while $T'=1/P$.
Therefore
\[
 T(w)-w=-\int_w^\infty\left(\frac1{P(s)}-1\right)\dd s,
 \qquad
 \frac1{b_q}\log(\e^{b_qw}-1)\le T(w)\le w.
\]
Since $V(w)=(D_2/\e)\e^{-2\delta T(w)}$, these inequalities give
\eqref{eq:planar-volume}.
\end{proof}

\begin{proposition}\label{prop:planar-moments}
If $1<q\le5$, or if $q>5$ and $\delta(q-2)\le1$, then
\[
 R_2\le\frac{9\pi(1-\log2)}{\e}<\frac{10}3,
 \qquad R_4\ge\frac{12}{\e}>4,
 \qquad \widehat Q>\frac{43}{54}>0.
\]
\end{proposition}
\begin{proof}
Layer-cake integration and \eqref{eq:planar-volume} give
\[
 R_2=\frac{8\delta^2}{D_2}\int_0^\infty wV(w)\dd w
 \le B(\delta/b_q),
 \qquad
 B(h)=\frac{8h^2}{\e}\int_0^\infty
                u(\e^u-1)^{-2h}\dd u.
\]
This integral is finite for $0<h<1$. On substituting $v=2hu$,
\[
 B(h)=\frac2\e\int_0^\infty
             v(\e^{v/(2h)}-1)^{-2h}\dd v.
\]
For fixed $v>0$, the integrand increases with $h$, because
\[
 \frac{\dd}{\dd a}\bigl[-a\log(\e^{v/a}-1)\bigr]
 =-\log(1-\e^{-v/a})+\frac{v/a}{\e^{v/a}-1}>0.
\]
If $q\le2$, then $\delta/b_q<(q+1)/4\le3/4$.
If $2\le q\le5$, the condition $q\delta<2$ gives
$\delta/b_q<(q+1)/(2q)\le3/4$.
For $q>5$ in the stated range,
$\delta/b_q=\delta(q-2)/2\le1/2$.
Hence $R_2\le B(3/4)$. The substitution
$\e^{-u}=\sin^2\theta$, followed by integration by parts, yields
\[
 \int_0^\infty u(\e^u-1)^{-3/2}\dd u
 =2\pi(1-\log2),
\]
and so $B(3/4)=9\pi(1-\log2)/\e<10/3$.
For example, the last strict inequality follows from
$\pi<22/7$, $\e>8/3$ and
$\log2>2(1/3+1/81)=56/81$.
Independently, the lower volume estimate gives
\[
 R_4=\frac{32\delta^4}{D_2}\int_0^\infty w^3V(w)\dd w
 \ge\frac{12}{\e}>4.
\]
Put $u=1+R_2<13/3$. For $0<\sigma<1$,
\[
 \frac{3\sigma^2-9\sigma+4}{6\sigma}\ge-\frac13,
 \qquad 1-\frac{4\sigma}{3}\ge-\frac13;
\]
the first inequality is equivalent to
$(3\sigma-4)(\sigma-1)\ge0$. Consequently
\[
 \widehat Q\ge\frac92-\frac{u^2+u}{3}+R_4
 >\frac{17}{2}-\frac{208}{27}=\frac{43}{54}>0.
\]
\end{proof}

\begin{proposition}\label{prop:planar-endpoint}
For $N=2$, $0<\beta<2$ and $1<q<4/(2-\beta)$,
$a_*<a_c<\infty$ and $\A_{a_c,2}$ is attained.
These conclusions hold for every admissible anisotropy $F$.
\end{proposition}
\begin{proof}
If $q\le5$, or $q>5$ and $\delta(q-2)\le1$, apply
Propositions~\ref{prop:planar-moments} and
\ref{prop:planar-optimized}. The remaining case has
$q>5$ and $\delta(q-2)>1$, so
Proposition~\ref{prop:planar-distant} applies.
This covers all the stated parameters. Proposition~\ref{prop:2.1}
transfers the conclusion to every admissible $F$: its exact reduction
preserves $W$, the normalized Green moments and the threshold.
\end{proof}

\subsection{Proofs of the main results}\label{sec:9.1}
Returning to general dimensions $N\ge2$, we now prove
Theorems~\ref{thm:1.2} and~\ref{thm:1.3} and deduce Corollary~\ref{cor:1.2}.
\begin{proposition}\label{prop:9.1}
If $q<q_-$, then
\begin{equation}\label{eq:116}
 \A_{a,N}>d(a)\qquad\text{for every }a>0.
\end{equation}
\end{proposition}
\begin{proof}
Fix $a>0$ and put $v_{a,\varepsilon}:=T_{N,a}\varphi_\varepsilon$.
Lemma~\ref{lem:3.3} gives $v_{a,\varepsilon}\in\Sa_a$ and, with
$x_\varepsilon=\norm{\varphi_\varepsilon}_q^N=O(Y_\varepsilon^{-1})$,
\begin{equation}\label{eq:114}
 \J(v_{a,\varepsilon})=
 \left(\frac{1-(1-x_\varepsilon)^{a/N}}{x_\varepsilon}\right)^\sigma
 \J(\varphi_\varepsilon).
\end{equation}
The factor has the expansion
\begin{equation}\label{eq:115}
 \left(\frac{1-(1-x_\varepsilon)^{a/N}}{x_\varepsilon}\right)^\sigma
 =\left(\frac{a}{N}\right)^\sigma
 \bigl[1+O_a(Y_\varepsilon^{-1})\bigr].
\end{equation}
Since $m/(N-1)<1$, \eqref{eq:109} implies
\[
 \J(v_{a,\varepsilon})\ge
 \left(\frac{a}{N}\right)^\sigma D_N+
 c_aY_\varepsilon^{-m/(N-1)}+O_a(Y_\varepsilon^{-1})
 >\left(\frac{a}{N}\right)^\sigma D_N
\]
for small $\varepsilon$, after decreasing $c_a>0$ if necessary.
By Lemma~\ref{lem:3.1} and Corollary~\ref{cor:7.3}, the last quantity is
at least $d(a)$.
\end{proof}
\begin{proof}[\textbf{\upshape Proof of Theorem~\ref{thm:1.2}}]
Proposition~\ref{prop:9.1} gives $\A_{a,N}>d(a)$ for every $a>0$.
Lemma~\ref{lem:2.2} then yields a nonnegative Wulff-symmetric maximizer.
\end{proof}

\label{sec:9.2}
\begin{proof}[\textbf{\upshape Proof of Theorem~\ref{thm:1.3}}]
Here $m=N-1$. Proposition~\ref{prop:8.2} gives $\A_{N,N}>d(N)$,
so Proposition~\ref{prop:3.5} yields the threshold $a_c$ defined by
\eqref{eq:29}. Theorem~\ref{thm:first-correction} and
Proposition~\ref{prop:scalar-endpoint} prove $a_c<\infty$.
Attainment at $a_c$ follows from Proposition~\ref{prop:highdim-endpoint}
when $N\ge3$ and Proposition~\ref{prop:planar-endpoint} when $N=2$.
Below $a_c$, the strict concentration gap and Lemma~\ref{lem:2.2}
give attainment; above $a_c$, Proposition~\ref{prop:3.5} gives nonattainment.
Independence of the anisotropy follows from Proposition~\ref{prop:2.1}.
\end{proof}

\label{sec:9.3}
\begin{proof}[\textbf{\upshape Proof of Corollary~\ref{cor:1.2}}]
Since $\beta>0$, $N<q_+$, and
$N<q_-\iff\beta>N/(N-1)$. Theorems~\ref{thm:1.1}--\ref{thm:1.3}
therefore give (i)--(iii), including nonnegative Wulff-symmetric
maximizers. The interval in (iii) is empty for $N=2$.
For $q=N$, Proposition~\ref{prop:2.1} has $c=1$ and
$\beta+q\delta=N$, so \eqref{eq:10} gives \eqref{eq:8} and independence
of $F$. For $N=q=2$, Proposition~\ref{prop:planar-endpoint} yields
$a_c>2[1+2(1+R_2)/\delta^2]$ with $R_2>0$.
Using $\delta=(2-\beta)/2$ gives the stated lower bound for each fixed
$\beta$, and hence divergence as $\beta\uparrow2$.
\end{proof}

\begin{remark}\label{rem:9.2}
For $q\ge q_+$, one has $m\ge N$, so the annular term
is $o(Y_\varepsilon^{-1})$ and the present first-order comparison no longer applies.
In this complementary range, \cite[Theorem~1.2]{ref12} gives a threshold
under the additional quantitative condition stated there. For $\beta>0$,
the lower threshold is zero: attainment holds below the upper threshold,
and nonattainment holds above it when it is finite. That theorem neither
establishes finiteness of the upper threshold nor decides attainment there.
\end{remark}

\appendix
\ifdefined\titleformat
\titleformat{\section}{\normalfont\Large\bfseries}{Appendix \thesection.}{1em}{}
\fi
\section{Asymptotic normalization of the Green-function test}\label{sec:A}
We prove the expansions and integral estimates in Lemma~\ref{lem:8.1}.
From the definition of $h_\varepsilon$, one has $h_\varepsilon=O(L_\varepsilon)$
on its inner branch. The unscaled outer energy is $t_\varepsilon+O(1)$,
the inner energy is $O(\log R)$, and the unscaled $L^q$ norm tends to
$\norm{G_*}_q$. Thus \eqref{eq:100} first gives $Y_\varepsilon\simeq L_\varepsilon$,
before any refined expansion is used. The nonlinear absorption term in
the Green equation contributes only exponentially small errors near the
pole. Throughout this appendix, $o_{\exp}(1)$ denotes a quantity bounded
by $CL_\varepsilon^A(R\varepsilon)^\vartheta$ for some fixed $A\ge0$ and
$\vartheta>0$; in particular, $o_{\exp}(1)=o(L_\varepsilon^{-M_0})$ for
every fixed $M_0>0$.

Let $G_*$ be the normalized Green function satisfying
\eqref{eq:90}--\eqref{eq:91}. Write $G_*(x)=g(r)$, $r=F^{o}(x)$.
In this appendix, $P$ and $\mathcal{P}$ denote the radial and level-set
fluxes of $G_*$, defined as in \eqref{eq:70} and \eqref{eq:72}.
The flux estimate \eqref{eq:71} also holds for $G_*$ and gives
\begin{equation}\label{eq:A.1}
 G_*(r)=-(N\kappa_N)^{-1/(N-1)}\log r+A_*+O\bigl(r^N|\log r|^{q-1}\bigr).
\end{equation}
If $r_t$ is defined by $G_*(r_t)=t$, then
\begin{equation}\label{eq:A.2}
 r_t^N=O(\e^{-\lambda_N t})
\end{equation}
and, for each fixed $s>0$,
\begin{equation}\label{eq:A.3}
 \int_{\{G_*>t\}}G_*^s\dd x=O(t^s\e^{-\lambda_N t}).
\end{equation}
We first compute the outer Dirichlet energy.
Testing \eqref{eq:91} on $\R^N\setminus W_r$ and using \eqref{eq:90} gives
\begin{equation}\label{eq:A.4}
 \int_{\{G_*<t\}}F(\nabla G_*)^N\dd x
 =t-c_q+O(t^q\e^{-\lambda_Nt}).
\end{equation}
Indeed, the error is
\[
 \norm{G_*}_q^{N-q}\int_{\{G_*>t\}}G_*^{q-1}(G_*-t)\dd x,
\]
which is controlled by \eqref{eq:A.3}.

On the outer linear branch of \eqref{eq:profile},
$\varphi_\varepsilon=Y_\varepsilon^{-1/N}G_*$. Hence
\begin{equation}\label{eq:A.5}
 \int_{\{G_*<t_\varepsilon\}}F(\nabla\varphi_\varepsilon)^N\dd x
 =\frac1{Y_\varepsilon}\left[\frac{N}{\lambda_N}
 (L_\varepsilon-\log R)-c_q+o_{\exp}(1)\right].
\end{equation}
On the inner branch, \eqref{eq:profile} gives
$h_\varepsilon'(t)=s(t)/(1+s(t))$, where
$s(t)=B\varepsilon^{-\gamma}\e^{-\lambda_\beta t/(N-1)}$.
The coarea formula and the flux representation yield
\begin{equation}\label{eq:A.6}
 \int_{\{G_*>t_\varepsilon\}}F(\nabla\varphi_\varepsilon)^N\dd x
 =\frac1{Y_\varepsilon}\int_{t_\varepsilon}^\infty
 \left(\frac{s(t)}{1+s(t)}\right)^N\mathcal{P}(t)\dd t.
\end{equation}
By \eqref{eq:71}, \eqref{eq:A.2} and \eqref{eq:72}, the error from replacing
$\mathcal{P}(t)$ by one is exponentially small. With the change of variables
$s=s(t)$,
\[
 \int_{t_\varepsilon}^\infty\left(\frac{s(t)}{1+s(t)}\right)^N\dd t
 =\frac{N-1}{\lambda_\beta}\int_0^{BR^\gamma}\frac{s^{N-1}}{(1+s)^N}\dd s.
\]
The asymptotic identity
\begin{equation}\label{eq:A.7}
 \int_0^L\frac{s^{N-1}}{(1+s)^N}\dd s
 =\log L-H_{N-1}+O(L^{-1})\qquad(L\to\infty)
\end{equation}
gives
\begin{equation}\label{eq:A.8}
\begin{aligned}
 \int_{\{G_*>t_\varepsilon\}}F(\nabla\varphi_\varepsilon)^N\dd x
 =\frac1{Y_\varepsilon}\bigg[&
 \frac{N}{\lambda_N}\log R+\frac1{\lambda_\beta}\log\frac{N\kappa_N}{N-\beta}\\
 &-\frac{N-1}{\lambda_\beta}H_{N-1}+O(R^{-\gamma})+o_{\exp}(1)\bigg].
\end{aligned}
\end{equation}
Combining \eqref{eq:A.5} and \eqref{eq:A.8},
\begin{equation}\label{eq:A.9}
\begin{aligned}
 \norm{F(\nabla\varphi_\varepsilon)}_N^N
 =\frac1{Y_\varepsilon}\bigg[&
 \frac{N}{\lambda_N}L_\varepsilon+\frac1{\lambda_\beta}\log\frac{N\kappa_N}{N-\beta}
 -\frac{N-1}{\lambda_\beta}H_{N-1}\\
 &-c_q+O(R^{-\gamma})+o_{\exp}(1)\bigg].
\end{aligned}
\end{equation}
On the outer branch, \eqref{eq:A.3} gives
\[
\int_{\{G_*<t_\varepsilon\}}|\varphi_\varepsilon|^q\dd x
 =Y_\varepsilon^{-q/N}\bigl(\norm{G_*}_q^q+o(Y_\varepsilon^{-M_0})\bigr)
\]
for every fixed $M_0>0$. The superlevel region has Euclidean volume
$O((R\varepsilon)^N)$, while $\varphi_\varepsilon=O(Y_\varepsilon^{(N-1)/N})$ there,
so its $L^q$ mass is exponentially small. Hence
\[
 \norm{\varphi_\varepsilon}_q^N
 =\frac{\norm{G_*}_q^N}{Y_\varepsilon}+o(Y_\varepsilon^{-2})
 =\frac{c_q}{Y_\varepsilon}+o(Y_\varepsilon^{-2}),
\]
which is \eqref{eq:106}.

Imposing \eqref{eq:101} in \eqref{eq:A.9} cancels $-c_q$ with the $L^q$
contribution and gives
\begin{equation}\label{eq:A.11}
\begin{aligned}
 Y_\varepsilon={}&\frac{N}{\lambda_N}L_\varepsilon+
 \frac1{\lambda_\beta}\log\frac{N\kappa_N}{N-\beta}
 -\frac{N-1}{\lambda_\beta}H_{N-1}\\
 &+O(R^{-\gamma})+o_{\exp}(1).
\end{aligned}
\end{equation}
Since $Y_\varepsilon\simeq L_\varepsilon$ and $R=L_\varepsilon^M$,
condition \eqref{eq:102} gives $R^{-\gamma}=o(Y_\varepsilon^{-1})$
and proves \eqref{eq:103}.

For the inner-region expansion, abbreviate
\[
 A_\varepsilon:=t_\varepsilon+\frac{N-1}{\lambda_\beta}
 \log(1+BR^\gamma)-Y_\varepsilon.
\]
Using \eqref{eq:103} and $R^{-\gamma}=o(Y_\varepsilon^{-1})$ gives
\begin{equation}\label{eq:A.12}
 A_\varepsilon=\frac{N-1}{\lambda_\beta}H_{N-1}+o(Y_\varepsilon^{-1}).
\end{equation}
It remains to estimate the weighted inner integral. For large $t$,
the radial flux is positive, so the inverse radius $r_t$ is differentiable.
Define the weighted volume by
\[
\begin{aligned}
 M(t)&:=\int_{\{G_*>t\}}F^{o}(x)^{-\beta}\dd x
      =\frac{N\kappa_N}{N-\beta}r_t^{N-\beta},\\
 -g'(r_t)&=\frac{\mathcal{P}(t)^{1/(N-1)}}{(N\kappa_N)^{1/(N-1)}r_t}.
\end{aligned}
\]
Differentiating the radial volume formula gives
\[
-M'(t)=\lambda_\beta\mathcal{P}(t)^{-1/(N-1)}M(t)
 \ge\lambda_\beta M(t),
\]
since $0<\mathcal{P}(t)\le1$. Thus $\e^{\lambda_\beta t}M(t)$ is
nonincreasing. From \eqref{eq:92},
\[
\lim_{t\to\infty}\e^{\lambda_\beta t}M(t)
 =\frac{N\kappa_N}{N-\beta}\e^{\lambda_\beta A_*}.
\]
Consequently,
\begin{equation}\label{eq:A.15}
 -M'(t)\ge (N\kappa_N)^{N/(N-1)}\e^{\lambda_\beta(A_*-t)}
\end{equation}
for all large $t$.

On the inner branch, write
\[
 \varphi_\varepsilon=Y_\varepsilon^{(N-1)/N}\left(1+\frac{k_\varepsilon(s)}{Y_\varepsilon}\right),
 \qquad
 k_\varepsilon(s):=A_\varepsilon-\frac{N-1}{\lambda_\beta}\log(1+s).
\]
Since $|k_\varepsilon(s)|=O(\log R)=o(Y_\varepsilon)$ uniformly for
$0\le s\le BR^\gamma$, the bracket is positive for small $\varepsilon$.
Convexity of $x\mapsto x^{\frac{N}{N-1}}$ gives
\[
 \varphi_\varepsilon^{\frac{N}{N-1}}\ge Y_\varepsilon+\frac{N}{N-1}k_\varepsilon(s).
\]
Using \eqref{eq:A.12}, we obtain
\[
 \e^{\lambda_\beta\varphi_\varepsilon^{\frac{N}{N-1}}}
 \ge\e^{\lambda_\beta Y_\varepsilon+NH_{N-1}+o(Y_\varepsilon^{-1})}(1+s)^{-N}.
\]
Using \eqref{eq:A.15}, the substitution
$s=B\varepsilon^{-\gamma}\e^{-\lambda_\beta t/(N-1)}$ and
\eqref{eq:103}, we obtain
\begin{equation}\label{eq:A.16}
 \int_{\{G_*>t_\varepsilon\}}
 \frac{\e^{\lambda_\beta\varphi_\varepsilon^{\frac{N}{N-1}}}}{F^{o}(x)^\beta}\dd x
 \ge D_N\e^{o(Y_\varepsilon^{-1})}
 \int_0^{BR^\gamma}(N-1)\frac{s^{N-2}}{(1+s)^N}\dd s.
\end{equation}
The density $(N-1)s^{N-2}(1+s)^{-N}$ has unit integral and tail
$O(R^{-\gamma})=o(Y_\varepsilon^{-1})$ beyond $BR^\gamma$. Hence
\begin{equation}\label{eq:A.17}
 \int_{\{G_*>t_\varepsilon\}}
 \frac{\e^{\lambda_\beta\varphi_\varepsilon^{\frac{N}{N-1}}}}{F^{o}(x)^\beta}\dd x
 \ge D_N+o(Y_\varepsilon^{-1}).
\end{equation}
It remains to replace the exponential by $\Ph$. Since $m$ is fixed,
the polynomial $\e^s-\Ph(s)$ has fixed degree $m-1$. The weighted measure
of the inner region is $O((R\varepsilon)^{N-\beta})$, and
$\lambda_\beta\varphi_\varepsilon^{\frac{N}{N-1}}=O(Y_\varepsilon)$ there. Its integral
is therefore exponentially small and, in particular, $o(Y_\varepsilon^{-1})$.
This proves \eqref{eq:104}. The identity \eqref{eq:105} follows directly
from the outer linear branch of \eqref{eq:profile}.


\begin{thebibliography}{31}
\small
\raggedright
\setlength{\itemsep}{3pt}
\addcontentsline{toc}{section}{References}
\bibitem{refAdachiTanaka2000}
S.~Adachi, K.~Tanaka, Trudinger type inequalities in $\R^N$ and their best
exponents, \textit{Proc. Amer. Math. Soc.} \textbf{128} (7) (2000) 2051--2057.
\href{https://doi.org/10.1090/S0002-9939-99-05180-1}{doi:10.1090/S0002-9939-99-05180-1}.

\bibitem{refAdimurthiSandeep2007}
Adimurthi, K.~Sandeep, A singular Moser--Trudinger embedding and its
applications, \textit{NoDEA Nonlinear Differential Equations Appl.}
\textbf{13} (5--6) (2007) 585--603.
\href{https://doi.org/10.1007/s00030-006-4025-9}{doi:10.1007/s00030-006-4025-9}.

\bibitem{ref2}
Adimurthi, Y.~Yang, An interpolation of Hardy inequality and Trudinger--Moser
inequality in $\R^N$ and its applications,
\textit{Int. Math. Res. Not. IMRN} 2010 (13) (2010) 2394--2426.
\href{https://doi.org/10.1093/imrn/rnp194}{doi:10.1093/imrn/rnp194}.

\bibitem{ref4}
A.~Alvino, V.~Ferone, G.~Trombetti, P.-L.~Lions, Convex symmetrization and
applications, \textit{Ann. Inst. H. Poincar\'e C Anal. Non Lin\'eaire}
\textbf{14} (2) (1997) 275--293.
\href{https://doi.org/10.1016/S0294-1449(97)80147-3}{doi:10.1016/S0294-1449(97)80147-3}.

\bibitem{refCarlesonChang1986}
L.~Carleson, S.-Y.~A.~Chang, On the existence of an extremal function for
an inequality of J.~Moser, \textit{Bull. Sci. Math. (2)}
\textbf{110} (2) (1986) 113--127.

\bibitem{refFlucher1992}
M.~Flucher, Extremal functions for the Trudinger--Moser inequality in
2 dimensions, \textit{Comment. Math. Helv.} \textbf{67} (1992) 471--497.
\href{https://doi.org/10.1007/BF02566514}{doi:10.1007/BF02566514}.

\bibitem{ref8}
K.~Guo, Y.~Liu, Sharp anisotropic singular Trudinger--Moser inequalities
in the entire space, \textit{Calc. Var. Partial Differential Equations}
\textbf{63} (2024) 82.
\href{https://doi.org/10.1007/s00526-024-02700-0}{doi:10.1007/s00526-024-02700-0}.

\bibitem{refGuoLiu2026Bounded}
K.~Guo, Y.~Liu, Singular Trudinger--Moser inequality involving $L^p$ norm
in bounded domain, \textit{Math. Nachr.} \textbf{299} (4) (2026) 948--976.
\href{https://doi.org/10.1002/mana.70124}{doi:10.1002/mana.70124}.

\bibitem{ref9}
K.~Guo, Y.~Liu, Existence of extremal functions and Wulff symmetry for
anisotropic Trudinger--Moser inequalities,
\textit{J. Geom. Anal.} \textbf{36} (2026) 258.
\href{https://doi.org/10.1007/s12220-026-02506-w}{doi:10.1007/s12220-026-02506-w}.

\bibitem{refIkomaIshiwataWadade2019}
N.~Ikoma, M.~Ishiwata, H.~Wadade, Existence and non-existence of maximizers
for the Moser--Trudinger type inequalities under inhomogeneous constraints,
\textit{Math. Ann.} \textbf{373} (1--2) (2019) 831--851.
\href{https://doi.org/10.1007/s00208-018-1709-5}{doi:10.1007/s00208-018-1709-5}.

\bibitem{refIshiwata2011}
M.~Ishiwata, Existence and nonexistence of maximizers for variational
problems associated with Trudinger--Moser type inequalities in $\R^N$,
\textit{Math. Ann.} \textbf{351} (4) (2011) 781--804.
\href{https://doi.org/10.1007/s00208-010-0618-z}{doi:10.1007/s00208-010-0618-z}.

\bibitem{refLamLuZhang2019}
N.~Lam, G.~Lu, L.~Zhang, Existence and nonexistence of extremal functions
for sharp Trudinger--Moser inequalities,
\textit{Adv. Math.} \textbf{352} (2019) 1253--1298.
\href{https://doi.org/10.1016/j.aim.2019.06.020}{doi:10.1016/j.aim.2019.06.020}.

\bibitem{ref3}
X.~Li, Y.~Yang, Extremal functions for singular Trudinger--Moser inequalities
in the entire Euclidean space, \textit{J. Differential Equations}
\textbf{264} (8) (2018) 4901--4943.
\href{https://doi.org/10.1016/j.jde.2017.12.028}{doi:10.1016/j.jde.2017.12.028}.

\bibitem{refLi2024AnisotropicSingular}
X.~Li, Anisotropic singular Trudinger--Moser inequalities on the whole
Euclidean space, \textit{Discrete Contin. Dyn. Syst.}
\textbf{44} (2) (2024) 462--490.
\href{https://doi.org/10.3934/dcds.2023111}{doi:10.3934/dcds.2023111}.

\bibitem{refLiRuf2008}
Y.~Li, B.~Ruf, A sharp Trudinger--Moser type inequality for unbounded
domains in $\R^n$, \textit{Indiana Univ. Math. J.}
\textbf{57} (1) (2008) 451--480.
\href{https://doi.org/10.1512/iumj.2008.57.3137}{doi:10.1512/iumj.2008.57.3137}.

\bibitem{ref12}
J.~Liang, Z.-C.~Xiong, On the thresholds of the critical singular anisotropic
Trudinger--Moser inequality, \textit{Forum Math.} \textbf{38} (6) (2026)
1919--1958.
\href{https://doi.org/10.1515/forum-2025-0226}{doi:10.1515/forum-2025-0226}.

\bibitem{ref7}
Y.~Liu, Concentration--compactness principle of singular Trudinger--Moser
inequality involving $N$-Finsler--Laplacian operator,
\textit{Int. J. Math.} \textbf{31} (11) (2020) 2050085.
\href{https://doi.org/10.1142/S0129167X20500858}{doi:10.1142/S0129167X20500858}.

\bibitem{refLiu2022}
Y.~Liu, Anisotropic Trudinger--Moser inequalities associated with the exact
growth in $\R^N$ and its maximizers, \textit{Math. Ann.}
\textbf{383} (3--4) (2022) 921--941.
\href{https://doi.org/10.1007/s00208-021-02194-7}{doi:10.1007/s00208-021-02194-7}.

\bibitem{refLiu2024Potential}
Y.~Liu, An improved Trudinger--Moser inequality involving
$N$-Finsler--Laplacian and $L^p$ norm,
\textit{Potential Anal.} \textbf{60} (2) (2024) 673--701.
\href{https://doi.org/10.1007/s11118-023-10066-9}{doi:10.1007/s11118-023-10066-9}.

\bibitem{ref10}
G.~Lu, Y.~Shen, J.~Xue, M.~Zhu, Weighted anisotropic isoperimetric
inequalities and existence of extremals for singular anisotropic
Trudinger--Moser inequalities, \textit{Adv. Math.} \textbf{458} (2024) 109949.
\href{https://doi.org/10.1016/j.aim.2024.109949}{doi:10.1016/j.aim.2024.109949}.

\bibitem{refManciniMartinazzi2017}
G.~Mancini, L.~Martinazzi, The Moser--Trudinger inequality and its extremals
on a disk via energy estimates, \textit{Calc. Var. Partial Differential Equations}
\textbf{56} (2017) 94.
\href{https://doi.org/10.1007/s00526-017-1184-y}{doi:10.1007/s00526-017-1184-y}.

\bibitem{refMasmoudiSani2015}
N.~Masmoudi, F.~Sani, Trudinger--Moser inequalities with the exact growth
condition in $\R^N$ and applications,
\textit{Comm. Partial Differential Equations} \textbf{40} (8) (2015) 1408--1440.
\href{https://doi.org/10.1080/03605302.2015.1026775}{doi:10.1080/03605302.2015.1026775}.

\bibitem{ref1}
J.~Moser, A sharp form of an inequality by N.~Trudinger,
\textit{Indiana Univ. Math. J.} \textbf{20} (1970--1971) 1077--1092.
\href{https://doi.org/10.1512/iumj.1971.20.20101}{doi:10.1512/iumj.1971.20.20101}.

\bibitem{ref11}
V.~H.~Nguyen, The thresholds of the existence of maximizers for the critical
sharp singular Moser--Trudinger inequality under constraints,
\textit{Math. Ann.} \textbf{380} (3--4) (2021) 1933--1958.
\href{https://doi.org/10.1007/s00208-020-02010-8}{doi:10.1007/s00208-020-02010-8}.

\bibitem{refRuf2005}
B.~Ruf, A sharp Trudinger--Moser type inequality for unbounded domains
in $\R^2$, \textit{J. Funct. Anal.} \textbf{219} (2) (2005) 340--367.
\href{https://doi.org/10.1016/j.jfa.2004.06.013}{doi:10.1016/j.jfa.2004.06.013}.

\bibitem{refTrudinger1967}
N.~S.~Trudinger, On imbeddings into Orlicz spaces and some applications,
\textit{J. Math. Mech.} \textbf{17} (5) (1967) 473--483.
\href{https://doi.org/10.1512/iumj.1968.17.17028}{doi:10.1512/iumj.1968.17.17028}.

\bibitem{refWangXia2012}
G.~Wang, C.~Xia, Blow-up analysis of a Finsler--Liouville equation in two
dimensions, \textit{J. Differential Equations} \textbf{252} (2) (2012) 1668--1700.
\href{https://doi.org/10.1016/j.jde.2011.08.001}{doi:10.1016/j.jde.2011.08.001}.

\bibitem{refWangChen2021}
X.~Wang, L.~Chen, Sharp weighted Trudinger--Moser inequalities with the
$L^n$ norm in the entire space $\R^n$ and existence of their extremal
functions, \textit{Potential Anal.} \textbf{54} (1) (2021) 153--181.
\href{https://doi.org/10.1007/s11118-019-09821-8}{doi:10.1007/s11118-019-09821-8}.

\bibitem{refXie2023}
R.~L.~Xie, Anisotropic Moser--Trudinger inequality involving $L^n$ norm in
the entire space $\R^n$, \textit{Acta Math. Sin. (Engl. Ser.)}
\textbf{39} (2023) 2427--2451.
\href{https://doi.org/10.1007/s10114-023-1692-y}{doi:10.1007/s10114-023-1692-y}.

\bibitem{ref5}
C.~Zhou, C.~Zhou, Moser--Trudinger inequality involving the anisotropic
Dirichlet norm $(\int_\Omega F^N(\nabla u)\dd x)^{1/N}$ on $W_0^{1,N}(\Omega)$,
\textit{J. Funct. Anal.} \textbf{276} (9) (2019) 2901--2935.
\href{https://doi.org/10.1016/j.jfa.2018.12.001}{doi:10.1016/j.jfa.2018.12.001}.

\bibitem{ref6}
C.~Zhou, C.~Zhou, On the anisotropic Moser--Trudinger inequality for
unbounded domains in $\R^n$, \textit{Discrete Contin. Dyn. Syst.}
\textbf{40} (2) (2020) 847--881.
\href{https://doi.org/10.3934/dcds.2020064}{doi:10.3934/dcds.2020064}.

\end{thebibliography}
\end{document}